\documentclass{amsart}
\usepackage{graphicx} 
\usepackage{mathrsfs}
\usepackage{hyperref}
\usepackage{amsmath, amsthm, amssymb, amscd}
\usepackage{enumerate}
\usepackage{hyperref}
\usepackage{verbatim}
\usepackage{cite}
\usepackage{color}

\DeclareMathOperator{\diam}{diam}

\theoremstyle{plain}
\newtheorem{theorem}{Theorem}

\newtheorem{corollary}[theorem]{Corollary}

\newtheorem{lemma}[theorem]{Lemma}

\newtheorem{proposition}[theorem]{Proposition}

\theoremstyle{definition}

\newtheorem{definition}[theorem]{Definition}
\newtheorem{remark}[theorem]{Remark}
\newtheorem{question}[theorem]{Question}

\numberwithin{equation}{section}
\numberwithin{theorem}{section}

\newcommand{\Gr}{\text{Gr}}
\DeclareMathOperator{\dist}{dist}

\newcommand{\RR}{\mathbb{R}}

\newcommand{\R}{\mathbb{R}}
\newcommand{\T}{\mathbb{T}}
\newcommand{\N}{\mathbb{N}}
\renewcommand{\S}{\mathbb{S}}

\newcommand{\g}{\gamma}
\newcommand{\G}{\Gamma}
\newcommand{\e}{\epsilon}
\renewcommand{\d}{\delta}


\title[Flatness, Menger curvature, and parametrization]{Flatness, Menger curvature, and parametrization}

\author{Guy C. David}
\address{Department of Mathematical Sciences\\ Ball State University\\ Muncie, IN 47306}
\email{gcdavid@bsu.edu}

\author{Vyron Vellis}
\address{Department of Mathematics\\ The University of Tennessee\\ Knoxville, TN 37966}
\email{vvellis@utk.edu}

\date{\today}

\thanks{G.C.~David was partially supported by NSF DMS grant 2054004. V.~Vellis was partially supported by NSF DMS grant 2154918.}


\begin{document}

\begin{abstract}
We show that on linearly locally contractible (LLC) manifolds, the beta numbers (which describe unilateral flatness) are comparable to the theta numbers (which describe bilateral flatness), quantitatively. As an application, we show that if $M\subset\R^n$ is a compact LLC $m$-manifold with finite Menger $p$-energy for some $p>m(m+2)$, then $M$ is in fact a $C^{1,\alpha}$ manifold. We also show that the bound $m(m+2)$ is critical by constructing, for each $n\geq 3$, an LLC $n$-sphere in $\R^{n+1}$ that has finite Menger $p$-energy for every $p<m(m+2)$ but is not even quasisymmetrically equivalent to the standard $n$-sphere.
\end{abstract}

\maketitle

\section{Introduction}

Many questions in geometric measure theory concern \emph{flatness} properties of sets in $\mathbb{R}^n$: how much does a given set resemble a linear subspace? Qualitatively, this gives rise to notions like tangent planes, and quantitatively to notions like the $\beta$ numbers of Jones \cite{Jones} that we will define below.

Roughly speaking, one can ask about ``unilateral'' flatness (``how close is your set to \emph{lying in} a plane?'') or ``bilateral'' flatness (``how close is your set to \emph{being} a plane?''). Much of this paper concerns the interplay between these two notions.

Following \cite{K} for convenience, we make the following definitions (although these notions are much older; see \cite{Jones, DS93}.)
\begin{definition}
Let $E\subseteq \RR^n$ and $1\leq m \leq n$. For a closed ball $B=B^n(x,r)$, write
$$ \beta^m_E(B) = \frac{1}{r}\inf_{P\in \Gr(m,n)} \sup\{\dist(z,x+P):z\in B^n(x,r)\cap E\}$$
and
$$ \theta^m_E(B) = \frac{1}{r} \inf_{P\in \Gr(m,n)}d_H((x+P)\cap B^n(x,r), E\cap B^n(x,r)). $$
Here $\Gr(m,n)$ is the Grassmannian of $m$-dimensional planes in $\RR^n$ and $d_H$ is the Hausdorff distance.
\end{definition}

Thus, smallness of $\beta^m_E(B)$ means that the set $E$ is ``unilaterally'' flat in $B$, while smallness of $\theta^m_E(B)$ means that the set $E$ is ``bilaterally'' flat in $B$.

Flatness considerations are intimately tied to questions of \emph{parametrization}. If a set has a good parametrization, this often forces it to be flat at many locations and scales; conversely, knowing that a set is flat at many locations and scales often implies a good parametrization of (parts of) the set. A few famous examples on this theme are:
\begin{itemize}
    \item The Analyst's Traveling Salesman Theorem of Jones \cite{Jones}, which says that control on a certain $\ell^2$ sum over all $\beta^1_E$ for a planar set $E$ is equivalent to $E$ being contained in a curve of finite length.
    \item Reifenberg's Topological Disk Theorem \cite{Reifenberg}, which says that a closed set in $\mathbb{R}^n$ with all numbers $\theta^m_E$ small (for balls below a given scale) is a H\"older disk.
\end{itemize}
For many other results in this vein, see \cite{DS93, ABT, DToro, DKT}, and the references therein.

Other authors (e.g., \cite{ABT}) use slightly different definitions of $\beta$ and $\theta$. For example, the definition of $\beta^m_E$ above corresponds to that of the ``centered'' $\beta_E^{\text{ctr}}$ in \cite{ABT}. It is easy to see that this is comparable up to a factor of $2$ to the notion used in \cite{ABT}; see \cite[Remark 3.2]{ABT}.

In addition, the authors of \cite{ABT} define their bilateral $\theta$ numbers as follows, writing $B=B^n(x,r)$. (We use $\tilde{\theta}$ to distinguish from our definition.)
$$ \tilde{\theta}^m_E(B) = \frac{1}{r} \inf_{P\in \Gr(m,n)}\max\left\{ \sup_{y\in E\cap B} \dist(y,x+P), \sup_{z\in (x+P)\cap B} \dist(z,E) \right\}. $$
It is not hard to see that for all balls $B\subset \R^n$
$$ \tilde{\theta}^m_E(B) \leq \theta^m_E(B) \leq 2\tilde{\theta}^m_E(B).$$ 

It is also simple to see that for all balls $B\subset \R^n$,
\begin{equation}\label{eq:beta<theta}
\beta^m_E(B) \leq \theta^m_E(B).
\end{equation}
On the other hand, the opposite inequality here is generally false and, as we show in the next two examples, $\beta_E^m$ and $\theta_E^m$ may not even be comparable. For instance, let $E=[0,\infty)\times\R\times\{0\}$ be a half-plane in $\R^3$ and let ${\bf 0}=(0,0,0)$ be the origin. Then it is easy to see that for all $r>0$,
\[ \beta^2_E(B^3({\bf 0},r)) =0 \quad\text{and}\quad \theta^2_E(B^3({\bf 0},r)) =1. \]
For a second example, let $F\subset\R^3$ be the graph of the function $f:\R^2 \to \R$ with $f(x) = \sqrt{|x|}$. An easy computation shows that for $r$ sufficiently small,
\[ \beta^2_F(B^3({\bf 0},r)) \leq r \quad\text{and}\quad \theta^2_F(B^3({\bf 0},r)) \geq 1/2. \]

In our first example above, the set $E$ is not a manifold but has no ``cusps'', while in the second example $M$ is a manifold but has a ``cusp''. Our first result below provides topological and geometric conditions under which the inequality \eqref{eq:beta<theta} \emph{can} be reversed, up to constants. The geometric condition we need is known as \emph{linear local contractibility}.


\begin{definition}
A set $E\subseteq \mathbb{R}^n$ is said to be \emph{linearly locally contractible} (\emph{LLC} for short) with parameters $C,R>0$ if for each $x\in E$ and $0<r<R$, the set $B^n(x,r)\cap E$ is contractible in $B^n(x,Cr)\cap E$. We often write \emph{$(C,R)$-LLC} to emphasize the parameters. 
\end{definition}

Our first main theorem is then the following, which says that inequality \eqref{eq:beta<theta} can be reversed, up to constants, for LLC manifolds. (This is closely connected to, and inspired by, the notion of ``fine'' sets from \cite{K}.)

\begin{theorem}\label{thm:LLCfine}
Given $n\in\N$ and constants $C,R>0$, there exists $K=K(C,R,n)$ and $R'= R'(C,R,n)$ such that if $M\subset \R^n$ is a closed $(C,R)$-LLC $m$-manifold, then
\begin{equation}\label{eq:fine2}
 \theta_M^m(B^n(x,r)) \leq K\beta_M^m(B^n(x,2r)),\quad\text{for all $x\in M$ and $0<r<R'$.}
\end{equation}
\end{theorem}

Theorem \ref{thm:LLCfine} follows from the sharper result Theorem \ref{thm:LLCfine2} below. We remark that there is nothing special about the factor 2 in the right-hand side of \eqref{eq:fine2}. In particular, following the arguments in the proof of Theorem \ref{thm:LLCfine}, one can show that $2r$ in \eqref{eq:fine2} can be replaced by $(1+\eta)r$ with $K$ and $R$ depending now on $\eta>0$.

As our second example above shows, the  LLC assumption cannot be removed from Theorem \ref{thm:LLCfine}. It is possible, however, that the LLC condition may be replaced by an assumption that all beta numbers are uniformly small. Following \cite{MattilaVuorinen}, given integers $1\leq m< n$, and numbers $\d\in (0,1)$ and $R>0$, we say that a set $X\subset \R^n$ has the $(m,\d,R)$\emph{-linear approximation property} (abbv. $(m,\d,R)$\emph{-LAP}) if  
\[ \beta_{X}^m(B^n(x,r))  < \delta, \quad\text{for all $x\in X$ and $r\in (0,R)$}. \]
The second author has posed the following question.

\begin{question}\label{q:lapreif}
Given positive integers $1\leq m\leq n-1$, is there a parameter $\delta_m \in (0,1)$ and a function $\phi_m:[0,1] \to [0,1]$ such that $\phi_m(t)\to0$ as $t\to0$ and the following property holds: If $M$ is a closed $m$-manifold in $\mathbb{R}^n$ with the $(m,\d,R)$-LAP for some $\d\in (0,\d_m)$ and $R>0$, then for all $x\in M$ and $r\in (0,R)$
\[ \theta^m_M(B^n(x,r)) \leq (1+\phi_m(\d))\beta^m_M(B^n(x,r)).\]
\end{question}
In other words, Question \ref{q:lapreif} asks, quantitatively, the following: \emph{If $M$ is an $m$-manifold with uniformly small $\beta^m$-numbers, must it have uniformly small $\theta^m$-numbers?}

We show that the answer is positive when $m=1$.


\begin{theorem}\label{prop:LAPfor1mfds2}
Let $\G$ be a topological circle in $\mathbb{R}^n$ with the $(1,\d,R)$-LAP for some $\delta \in (0,\frac1{8\sqrt{2}})$ and $R>0$. Then for all $x\in\G$ and all $r\in (0,R)$ we have
\[ \theta^1_{\G}(B^n(x,r)) \leq (1+\tfrac12\delta^2)\beta_{\G}^1(B^n(x,r))\]
\end{theorem}

Along the way of proving Theorem \ref{prop:LAPfor1mfds2}, we show that topological circles with the $(1,\d,R)$-LAP for some $\delta \in (0,\frac1{8\sqrt{2}})$ are LLC in a quantitative fashion; see Proposition \ref{prop:LAPimpliesBT}. 

We remark that not only is Question \ref{q:lapreif} still open, to the best of our knowledge even a less quantitative variant is still open, namely:
\begin{question}\label{q:weaklapreif}
Given positive integers $1\leq m\leq n-1$, and $\epsilon>0$, is there a parameter $\delta=\delta(m,n,\epsilon)$ with the following property:

If $M$ is a closed $m$-manifold in $\mathbb{R}^n$ with the $(m,\d,R)$-LAP for some $R>0$, then $M$ is $(m,\epsilon, R')$-Reifenberg flat for some $R'>0$.
\end{question}

Here a closed set $X$ is $(m,\epsilon, R')$-Reifenberg flat if 
$$ \theta^m_X(B^n(x,r)) < \epsilon \text{ for all } x\in X \text{ and } r\in (0,R').$$

A positive answer to Question \ref{q:lapreif} would imply a positive answer to Question \ref{q:weaklapreif}, and in particular both have a positive answer when $m=1$.

We next give some applications of Theorem \ref{thm:LLCfine}.

\subsection{Menger curvature and parametrization}

The first application of Theorem \ref{thm:LLCfine} is towards the \emph{quasisymmetric uniformization problem}. Informally speaking, quasisymmetric mappings are homeomorphisms which distort shapes by a bounded amount, compared to bi-Lipschitz mappings which distort shapes and sizes by a bounded amount; see Section \ref{sec:prelim} for precise definition. 

The problem of classifying all metrics $d$ on $\mathbb{S}^n$ for which $(\S^n,d)$ is quasisymmetrically homeomorphic to $\S^n$ is a major open question in non-smooth analysis and can be thought as a relaxation of the Beltrami problem in the analytic theory of quasiconformal mappings; see \cite[Open problem 15.6]{Heinonen}. While completely solved when $n=1$ \cite{TuVa,HerMey}, the uniformization problem turns out to be formidable in dimensions $n\geq 2$. A breakthrough in this problem when $n=2$ was given by Bonk and Kleiner \cite{BonK} (see also \cite{LyWe,Raj}) who showed that if $(\S^2,d)$ is LLC and \emph{Ahlfors $2$-regular}, then it is quasisymmetric equivalent to $\mathbb{S}^2$. Ahlfors $2$-regularity 
asserts that there exists a measure $\mu$ on $X$ such that each ball of radius $r$ with $r$ sufficiently small has measure comparable to $r^2$. Examples of Semmes \cite{Semmes1996}, Heinonen and Wu \cite{HW}, and of the second named author and Pankka \cite{PV} show that the LLC condition and Ahlfors regularity are insufficient in dimensions $n\geq 3$. A rich collection of intriguing examples of ``quasisymmetric spheres'' exists in literature \cite{Lewis,bishop, DToro, Meyer,PW, RV,Wu,Wu2,VW2,V1}. 

The bi-Lipschitz uniformization problem is also interesting since it is intimately related to the \emph{Quasiconformal Jacobian Problem} \cite{BoHeSa}, i.e., the characterization of weights in $\R^n$ that are comparable to the Jacobian of a quasiconformal homeomorphism of $\R^n$. Sufficient conditions for local bi-Lipschitz parametrization have been given in terms of bounds of local flatness \cite{Toro2, DavidToro}, existence of flat forms in Sobolev spaces \cite{HeinonenSullivan,HeinonenKeith} and integral bounds of the curvature tensor \cite{BoLa, Fu2, MuSve, Toro}.

Here we consider the \emph{Menger curvature}, a discrete notion of curvature conceived by Menger in the 1970s that has played a very important role in geometric measure theory and, especially, higher dimensional rectifiability \cite{Leger}. In this article, we use a version of Menger curvature proposed by Lerman and Whitehouse \cite{LeWh}.

\begin{definition}
Given $n\in\mathbb{N}$, $k\in\{1,\dots,n-1\}$, and points $x_1,\dots,x_{k+2} \in \mathbb{R}^n$ that are not all the same, define
\[ \mathcal{K}(x_0,\dots,x_{m+1}) := \frac{\mathcal{H}^{m+1}(\Delta(x_0,\dots,x_{m+1}))}{(\diam{\Delta(x_0,\dots,x_{m+1})})^{m+2}},\]
where $\Delta(x_0,\dots,x_{m+1})$ is the $(m+2)$-simplex formed by the points $x_0,\dots,x_{m+1}$. If $x_1=\cdots = x_{k+2}$, then we set $\mathcal{K}(x_1,\dots,x_{k+2}) = 0$. 
\end{definition}

Fix a set $X\subset \mathbb{R}^n$ and a parameter $m\in \{0,1, \dots, n-1\}$, which we think of as the ``dimension'' of $X$. Following \cite{K}, for $p>0$, we define the geometric curvature energy
$$ \mathcal{E}^{m+2}_p(X) = \int_{X^{m+2}} \mathcal{K}(x_0, \dots, x_{m+1})^p \, d\mathcal{H}^{m(m+2)}(x_0,\dots,x_{m+1}).$$

By combining Theorem \ref{thm:LLCfine} with the results of Kolasi\'nski \cite{K} and David, Kenig, and Toro \cite{DKT}, we obtain the following:

\begin{theorem}\label{thm:kolasinski2}
Let $M$ be a compact LLC $m$-manifold in $\RR^n$. Suppose that 
$$ \mathcal{E}^{m+2}_p(M) < \infty$$
for some $p>m(m+2)$. Then $M$ is a $C^{1,\alpha}$ manifold for some $\alpha=\alpha(p,m)$.

In particular, if $M$ is a topological sphere, then $M$ is bi-Lipschitz to the standard $m$-sphere $\mathbb{S}^m$.
\end{theorem}

Theorem \ref{thm:kolasinski2} follows from the more general result Theorem \ref{cor:kolasinski} which also provides the value of $\alpha$. We remark that, if Question \ref{q:lapreif} has an affirmative answer, then the LLC assumption can be removed from Theorem \ref{thm:kolasinski2}.



By adapting an example of Pankka and the second author \cite{PW}, we demonstrate the sharpness of condition $p>m(m+2)$ in Theorem \ref{cor:kolasinski} for bi-Lipschitz (or, even weaker, quasisymmetric) parametrizations.

\begin{theorem}\label{thm:ex}
For each $n\geq 3$ there exists an Ahlfors $n$-regular and LLC topological $n$-sphere $\Sigma_n \subset \mathbb{R}^{n+1}$ such that $\mathcal{E}_{p}^{n+2}(\Sigma_n) < \infty$ for all $p<n(n+2)$, but
$\Sigma_n$ is not quasisymmetrically equivalent to $\mathbb{S}^n$.
\end{theorem}

It would be interesting to understand what happens at the critical case $p=m(m+2)$. More specifically, if $M$ is a compact LLC $m$-manifold in $\RR^n$ with $\mathcal{E}^{m+2}_{m(m+2)}(M) < \infty$, does $M$ admit local quasisymmetric (or even bi-Lipschitz) parametrizations by $\R^m$?

\subsection{Quasiplanes}

Another application of Theorem \ref{thm:LLCfine} concerns flatness properties of quasiplanes (quasisymmetric images of $\mathbb{R}^n$ in $\mathbb{R}^N$). This topic, and its relation to local bi-Lipschitz parametrizations, is the subject of the papers \cite{BGRT,ABT}. In \cite{ABT}, the following lemma is proven. 

\begin{lemma}[{\cite[Lemma 4.2]{ABT}}]\label{lem:ABT}
Let $1\leq n \leq N-1$, let $V\subset\R^N$ be an $n$-dimensional plane, let $v\in V$, and let $e \in \mathbb{S}^{N-1}$ be a unit vector in $\mathbb{R}^N$. If $f\colon B^N(v,2r) \rightarrow\mathbb{R}^N$ is a topological embedding, then
\begin{equation}\label{eq:abt}
\beta_{f(V\cap B^N(v,2r))}^n\left( B^N(f(v), \tfrac12 |f(v+re)-f(v)|)\right) \leq 72N (H_f(B^N(v,2r))-1).\end{equation}
\end{lemma}
Here
$$ H_f(X) = \inf\{ H : f \text{ is weakly } H\text{-quasisymmetric on }X\}.$$
See Section \ref{sec:prelim} for the definition of weakly quasisymmetric maps. In other words, this lemma says that if a quasisymmetry is ``close to a similarity'', then its image has small $\beta$-numbers.

In the special case $n=N-1$, the authors of \cite{BGRT}, adapting earlier work of Prause \cite{Prause}, obtain a similar but stronger result with the important difference that $\theta$ rather than $\beta$ appears on the left hand side of \eqref{eq:abt}. (See \cite[Lemma 2.3]{BGRT} or \cite[Lemma 4.1]{ABT}.) In the case of higher co-dimension, the authors of \cite{ABT} remark that they were unable to generalize the two-sided flatness bound from co-dimension $1$ to higher co-dimension; see the discussion in \cite[p. 199]{ABT}.

On the other hand, Theorem \ref{thm:LLCfine} allows for an immediate upgrading of Lemma \ref{lem:ABT} to a bilateral estimate in arbitrary co-dimension.

\begin{corollary}\label{lem:ABT2}
Let $1\leq n \leq N-1$, let $V\subset \R^N$ be an $n$-dimensional plane, let $v\in V$, and let $e\in \mathbb{S}^{N-1}$ be a unit vector in $\mathbb{R}^N$. If $f\colon B^N(v,2r) \rightarrow\mathbb{R}^N$ is a topological embedding, then
$$ \theta_{f(V\cap B^N(v,2r))}^n\left (B^N(f(v), \tfrac14 |f(v+re)-f(v)|)\right) \lesssim_N H_f(B^N(v,2r))-1.$$
\end{corollary}

We prove Corollary \ref{lem:ABT2} in Section \ref{sec:quasiplanes}. It would be interesting to have analogous result in the weaker case where $f$ is defined not on all of $B^{N}(v,2r)$ but just on $V\cap B^{N}(v,2r)$.

\subsection{Outline of the paper}
After giving some preliminaries in Section \ref{sec:prelim}, we prove Theorem \ref{thm:LLCfine} in Section \ref{sec:LLCfine}. In Section \ref{sec:quasiplanes} we use Theorem \ref{thm:LLCfine} to prove Corollary \ref{lem:ABT2}. The parametrization Theorem \ref{thm:kolasinski2} is proven in Section \ref{sec:mengerparam} and the accompanying example Theorem \ref{thm:ex} is proven in Section \ref{sec:ex}. Finally, the proof of Theorem \ref{prop:LAPfor1mfds2}, and some related results, occupy Section \ref{sec:1dim}.

\subsection*{Acknowledgments} The ideas in the proof of Theorem \ref{prop:LAPfor1mfds2} and Proposition \ref{prop:LAPimpliesBT} were developed in discussions with Matthew Badger. The second author also thanks Ball State University for its hospitality.

\section{Notation and preliminaries}\label{sec:prelim}
We write $B(x,r)$ for a closed ball of radius $r$ centered at $x$. Given a set $A\subset \R^n$ and a number $\e>0$, we denote by
\[ N_{\epsilon}(A) = \{y\in \R^n : \dist(y,A) \leq \e\}\]
the closed $\epsilon$-neighborhood of $A$.

An embedding $f: (X,d) \to (Y,d')$ between metric spaces is said to be \emph{quasisymmetric} if there exists a homeomorphism $\eta: [0,+\infty) \to [0,+\infty)$ such that for all $x,a,b \in X$ with $x\neq b$, 
\[ \frac{d'(f(x),f(a))}{d'(f(x),f(b))} \leq \eta\left( \frac{d(x,a)}{d(x,b)}\right). \] 

Quasisymmetric mappings preserve many geometric properties but need not be smooth. Since their introduction by Beurling and Ahlfors \cite{BerAhl}, quasisymmetric maps have played an important role in the development of geometric function theory as they are considered the right generalization of conformality to a general metric space setting. 

It is well-known that on certain metric spaces, quasisymmetry is equivalent to a weaker condition, aptly named \emph{weak quasisymmetry}. We say that a homeomorphism $f: (X,d) \to (Y,d')$ between metric spaces is said to be weakly $H$-\emph{quasisymmetric}, for some $H\geq 1$, if for any triple $x,a,b\in X$ 
\[ d(x,a)\leq d(x,b) \quad\text{implies}\quad d'(f(x),f(a))\leq H d'(f(x),f(b)).\]

It almost immediately follows from the definition that if $f$ is quasisymmetric with control $\eta$, then it is weakly $\eta(1)$-quasisymmetric. As the next lemma states, the converse is also true for connected doubling metric spaces, including all Euclidean spaces and balls therein.

\begin{lemma}[{\cite[Theorem 10.19]{Heinonen}}]\label{lem:weaktoQS}
Let $X$ be a connected doubling metric space and let $Y$ be a doubling metric space. A weakly $H$-quasisymmetric embedding $f:X \to Y$ is quasisymmetric, quantitatively.
\end{lemma}

\begin{remark}\label{rem:similarities}
Recall that a mapping $f:(X,d)\to (Y,d')$ is a \emph{similarity} if there exists $\lambda>0$ (called the \emph{scaling factor}) such that $d'(f(x),f(y)) = \lambda d(x,y)$ for all $x,y \in X$. The following statements are easy and are left to the reader.
\begin{enumerate}
\item If $f:(X,d)\to (Y,d')$ is a similarity, then it is weakly 1-quasisymmetric.
\item If $f:\R^n\to\R^n$ is weakly $1$-quasisymmetric, then it is a similarity.
\item If $f:(X,d)\to (Y,d')$ is weakly $H$-quasisymmetric and if $\zeta_1 : (X,d)\to (X,d)$ and $\zeta_2 : (Y,d') \to (Y,d')$ are similarities, then 
\[\zeta_2 \circ f\circ\zeta_1 : (X,d) \to (Y,d')\] 
is weakly $H$-quasisymmetric.
\end{enumerate}
\end{remark}

See \cite{Heinonen} for a detailed exposition on the theory of quasisymmetric maps.

We will also need a notion of pointed Hausdorff convergence for (possibly unbounded) pointed subsets of Euclidean space. The following definition from \cite[Chapter 8]{DS97} is convenient:

\begin{definition}\label{def:ptdhausdorff}
Let $\{F_j\}$ be a sequence of non-empty closed subsets of $\mathbb{R}^n$, and $F$ another non-empty closed subset of $\mathbb{R}^n$. We say that $\{F_j\}$ \emph{converges} to $F$ if the following hold for all $R>0$:
$$\lim_{j\rightarrow \infty} \sup_{x\in F_j \cap B(0,R)} \dist(x,F) = 0$$
and
$$\lim_{j\rightarrow \infty} \sup_{x\in F \cap B(0,R)} \dist(x,F_j) = 0.$$
\end{definition}

We will rely on the following basic compactness lemma for this notion of convergence:
\begin{lemma}[Lemma 8.2 of \cite{DS97}]
Let $\{F_j\}$ be a sequence of non-empty closed subsets of $\mathbb{R}^n$ that each intersect a fixed ball $B=B(0,r)\subset\mathbb{R}^n$. Then there is a subsequence of $\{F_j\}$ that converges to a non-empty closed subset $F$ of $\mathbb{R}^n$.
\end{lemma}

Recall that given a metric space $X$ and some $s>0$, the \emph{(outer) Hausdorff $s$-measure} is defined by
\begin{align*} \mathcal{H}^s(A) =& \\
\lim_{\delta\to0} &\inf\left\{ \sum_{i=1}^{\infty}(\diam{B_i})^s : B_i\text{ are balls with $\diam{B_i} \leq \d$ and $A\subset \bigcup_{i=1}^{\infty}B_i$} \right\}.
\end{align*}
When restricted on the Borel $\sigma$-algebra of $X$, $\mathcal{H}^s$ is an actual measure. Moreover, the Hausdorff measure has the following scaling property: if $X\subset\R^n$ is a Borel set, $s>0$, and $\phi:\R^n\to\R^n$ is a similarity with scaling factor $\lambda$, then
\begin{equation}\label{eq:Hausscale}
\mathcal{H}^s(\phi(X)) = \lambda^s\mathcal{H}^s(X).
\end{equation}

Finally, we say that a metric space $X$ is \emph{(Ahlfors) $s$-regular} for some $s> 0$ if there exists a Borel measure $\mu$ on $X$ and a constant $C\geq 1$ such that
\begin{equation}\label{eq:reg}
C^{-1}r^s \leq \mu(B(x,r)) \leq C r^s, \quad\text{for all $x\in X$ and $r\in (0,\diam{X})$}.   
\end{equation} 
It is not hard to see that if a metric space $X$ is $s$-regular, then it satisfies \eqref{eq:reg} with $\mu$ replaced by $\mathcal{H}^s$ and $C$ replaced by some (potentially different) constant $C'\geq 1$; see \cite[\textsection1.4.3]{MT10}.

\section{LLC manifolds}\label{sec:LLCfine}

In this section we prove the following stronger version of Theorem \ref{thm:LLCfine}.

\begin{theorem}\label{thm:LLCfine2}
Given $n\in\N$ and constants $C,R>0$, there exist $R'= R'(C,R,n)$ and a threshold $\delta=\delta(C,R,n)$ with the following property. If $M$ is a closed $(C,R)$-LLC $m$-manifold in $\RR^n$, with $\beta^m_M(B^n(x,2r))\leq \delta$ for all $x\in M$ and $r\in (0,R')$, then
\begin{equation}\label{eq:fine} 
\theta_M^m(B^n(x,r)) \leq 4\beta_M^m(B^n(x,2r)).
\end{equation}
\end{theorem}

We require some definitions and preliminary results from algebraic topology.

\begin{definition}
A \emph{Euclidean neighborhood retract} is a set in some Euclidean space that is a retract of one of its open neighborhoods. If $M\subseteq \mathbb{R}^n$ is a Euclidean neighborhood retract and satisfies the (integer singular) homology condition
$$ H_*(M, M\setminus \{x\}) \cong H_*(\RR^m, \RR^m \setminus\{x\}),$$
it is called a \emph{homology $m$-manifold}.
\end{definition}

All $m$-manifolds are homology $m$-manifolds, but there are examples of homology $m$-manifolds that are not manifolds \cite{Poinc}. We will need the following lemma, which is a consequence of the fact that the invariance of domain principle continues to hold for homology manifolds. Denote by $\textbf{0}$ the origin in $\R^n$.

\begin{lemma}\label{lem:homologyball}
Let $M$ be a closed subset of $\RR^n$ that is an LLC homology $m$-manifold. Suppose that
$$ \emptyset \neq M\cap B^n({\bf 0},1) \subset P\cap B^n({\bf 0},1)$$
for some $m$-plane $P\subset\RR^n$ passing through the origin.

Then 
$$ M\cap B^n({\bf 0},1) = P\cap B^n({\bf 0},1)$$
\end{lemma}
\begin{proof}
We first assume that $M$ is connected. In this case, the lemma follows from two topological facts. We use the terminology of Bredon's book \cite{Bredon} to state these.
\begin{itemize}
\item If $M$ is a connected LLC homology $m$-manifold, then it is an ``$m-cm_{\mathbb{Z}}$'' (a cohomology $m$-manifold); see the implication (c)$\Rightarrow$(a) in \cite[Theorem V.16.8]{Bredon}. The same therefore holds for open subsets of $M$. 
\item Spaces of this type satisfy invariance of domain: if $A\subseteq B$ are both $m-cm_{\mathbb{Z}}$ spaces, then $A$ is open in $B$. See \cite[Corollary V.16.19]{Bredon}.
\end{itemize}

For our lemma, observe by the first fact above that $M\cap B^n({\bf 0},1)$ and $P\cap B^n({\bf 0},1)$ are both $m-cm_{\mathbb{Z}}$ spaces. It follows by the second fact that $M\cap B^n({\bf 0},1)$ is open in $P \cap B^n({\bf 0},1)$. On the other hand, $M\cap B^n({\bf 0},1)$ is closed in $P\cap B^n({\bf 0},1)$ since $M$ is closed. Connectedness of $P\cap B^n({\bf 0},1)$ then implies that $M\cap B^n({\bf 0},1)=P\cap B^n({\bf 0},1)$.

If $M$ is not connected, it is a union of connected components, each of which are LLC homology $m$-manifolds. The argument above then applies to each component and yields the same conclusion.
\end{proof}

\subsection{Preliminary version of Theorem \ref{thm:LLCfine2}}

We start by proving a weaker version of Theorem \ref{thm:LLCfine2}.

\begin{theorem}\label{thm:LLCflat2}
Let $C>0$ and $m,n\in\mathbb{N}$. For each $\epsilon>0$, there is a $\delta=\delta(\epsilon,m,n,C)>0$ with the following property:

If $M$ is a $(C,R)$-LLC $m$-manifold in $\RR^n$ and for some $x\in M, r<R$ there is an $m$-plane $P$ with
$$ \sup_{z\in M\cap B^n(x,r)} \dist(z, P \cap B^n(x,r)) \leq \delta r,$$
 then
 $$ \sup_{z\in P\cap B^n(x,r)} \dist(z, M \cap B^n(x,r)) \leq \epsilon r.$$
\end{theorem}
Theorem \ref{thm:LLCflat2} essentially says that (for LLC manifolds) if $\beta$ is small, then $\theta$ is small, but without linear control.

\begin{proof}[Proof of Theorem \ref{thm:LLCflat2}]
Suppose the conclusion were false. Then there would exist positive constants $C, m, n,\epsilon$ with the following property: For each $k\in\mathbb{N}$, there is an $R_k>0$ and an $m$-manifold $M_k\subset \mathbb{R}^n$ that is $(C,R_k)$-LLC and fails the conclusion of the theorem for $\delta=\frac{1}{k}$. 

Thus, each $M_k$ contains a ball $B^n(x_k, r_k)$, with $r_k< R_k$, for which there is an $m$-plane $P_k$ satisfying
\begin{equation}\label{eq:Mkflat}
\sup_{z\in M_k\cap B^n(x_k,r_k)} \dist(z, P_k \cap B^n(x_k,r_k)) \leq \frac1k r_k.
\end{equation}
and
\begin{equation}\label{eq:Mknotflat}
\sup_{z\in P_k\cap B^n(x_k,r_k)} \dist(z, M_k \cap B^n(x_k,r_k))> \epsilon r_k.
\end{equation}

Without loss of generality, we may translate so that each $x_k=0\in\RR^n$. 
Let $X_k = \frac{1}{r_k} M_k$. Then each $X_k$ is a $(C,1)$-LLC topological manifold (since $R_k/r_k > 1$).

Equations \eqref{eq:Mkflat} and \eqref{eq:Mknotflat} translate to:
\begin{equation}\label{eq:Xkflat}
\sup_{z\in X_k\cap B^n({\bf 0},1)} \dist(z, P_k \cap B^n({\bf 0},1)) \leq \frac1k.
\end{equation}
and
\begin{equation}\label{eq:Xknotflat}
\sup_{z\in P_k\cap B^n({\bf 0},1)} \dist(z, X_k \cap B^n({\bf 0},1))> \epsilon.
\end{equation}

We now pass to an appropriate subsequence. For simplicity, we will continue to index this subsequence by $\{X_k\}$, $\{P_k\}$, etc. We may choose the subsequence so that:
\begin{enumerate}[(i)]
\item The sequence $\{X_k\}$  converges in the sense of Definition \ref{def:ptdhausdorff} to a closed set $X\subset \RR^n$. 
\item The sequence of $m$-planes $\{P_k\}$ converges in the pointed Hausdorff sense to an $m$-plane $P$. 
\item There are points $z_k \in P_k \cap B^n({\bf 0},1)$ with $\dist(z_k, X_k\cap B^n({\bf 0},1))>\epsilon$ and such that
$$ z_k \rightarrow z$$
for some point $z\in P \cap B^n({\bf 0},1).$
\end{enumerate}
Note that \eqref{eq:Xkflat} and the fact that $0\in X_k$ imply that $P$ passes through the origin.

By \cite[Proposition 2.19]{GCD}, the closed set $X$ is in fact an LLC homology $m$-manifold. (Note that, while the stated assumptions of that result include Ahlfors regularity, this is not required here. The Ahlfors regularity assumption in \cite[Proposition 2.19]{GCD} is used only to obtain the fact that the (Gromov-)Hausdorff limit has finite topological dimension. In our case, that is automatic, as the limit is a subset of $\mathbb{R}^n$.)

Equation \eqref{eq:Xkflat} implies that 
$$X \cap B^n({\bf 0},1) \subset P\cap B^n({\bf 0},1).$$

Item (iii) implies that the point $z\in P \cap B^n({\bf 0},1)$ is at distance at least $\epsilon$ from $X$. It follows that the point point $z' = (1-\epsilon/10)z$ is in $P\cap B^n({\bf 0},1)$ but not in $X$. Thus, $X\cap B^n({\bf 0},1)$ is a proper subset of $P\cap B^n({\bf 0},1)$. This is impossible, by Lemma \ref{lem:homologyball}.
\end{proof}

\subsection{Proof of Theorem \ref{thm:LLCfine2}}
We require the following fact about LLC sets due to Borsuk.

\begin{lemma}[{\cite[Section 13]{Borsuk}}]\label{lem:retraction}
If $E$ is a closed $(C,R)$-LLC subset of $\RR^n$, then there are $s=s(C,R)$ and $L=L(C,R)$ with the following property:

There is a continuous retraction
$$ \phi: N_{s}(E) \rightarrow E$$
such that
$$ |\phi(x)-x| \leq L\dist(x,E)$$
for all $x\in N_{s}(E)$.
\end{lemma}

We are now ready to prove Theorem \ref{thm:LLCfine2}.

\begin{proof}[Proof of Theorem \ref{thm:LLCfine2}]
Fix constants $C\geq 1$ and $R>0$. Obtain parameters $s=s(C,R)>0$ and $L=L(C,R)\geq 1$ from Lemma \ref{lem:retraction}. Choose $\epsilon>0$ sufficiently small depending on these; to be specific, we require that $\epsilon < \frac{1}{10L}$. Obtain $\delta=\delta(\epsilon)>0$ from Theorem \ref{thm:LLCflat2}; we may assume that $\delta< \frac{1}{10}$.

Fix a closed $(C,R)$-LLC $m$-manifold $M\subseteq \RR^n$. By Lemma \ref{lem:retraction}, we have a retraction $\phi\colon N_s(M)\rightarrow M$ such that
\begin{equation}\label{eq:retraction}
 |\phi(q)-q| \leq L\dist(q,M) \text{ for all } q\in N_{s}(M).
\end{equation}

Fix a ball $B=B^n(x,r)$ with $x\in M$ and $0<r<R':=s/2$. Let
$$ \beta := \beta^m_M(2B) \quad\text{and}\quad \theta:=\theta^m_M(B).$$
Our claim is that if $\beta<\delta$, then $\theta \leq 4\beta$. We may assume during the proof that $\beta>0$; the $\beta=0$ case follows by a limiting argument.


Suppose $\beta <\delta$, and let $P$ be an $m$-plane through $x$ with 
\begin{equation}\label{eq:MnearP}
 M\cap 2B \subseteq N_{2\beta r}(P).
\end{equation}

Proving that $\theta\leq 4\beta$ requires showing that $M\cap B$ is in the closed $4\beta r$-neighborhood of $P\cap B$, and, vice versa, that $P\cap B$ is in the closed $4\beta r$-neighborhood of $M\cap B$.

The former is easy: if $q\in M\cap B^n(x,r) \subseteq M \cap B^n(x,2r)$, then $q$ is within $2\beta r$ of a point of $P$. This point may be taken to be the orthogonal projection of $q$ onto $P$, which also lies in $P\cap B^n(x,r)$.

For the remainder of the proof, we work to prove that 
$$ P\cap B \subseteq N_{4\beta r}(M\cap  B).$$

Since $\beta<\delta$, Theorem \ref{thm:LLCflat2} implies that 
\begin{equation}\label{eq:PnearM}
 P \cap 2B \subseteq N_{2\epsilon r}(M\cap 2B).
\end{equation}

Let $B'= B^n(x,1.5r)$. Consider the map
$$ h = p\circ \phi\colon P\cap B' \rightarrow P, $$
where $p$ denotes the orthogonal projection to $P$.

Note that if $z\in P\cap B'$, then
$$ |\phi(z)-z| \leq L\dist(z,M) \leq 2L\epsilon r,$$
by \eqref{eq:PnearM}. Thus, $\phi(z)\in 2B$, since $2L\epsilon < 0.5$. It follows from \eqref{eq:MnearP} that the projection $p$ from $\phi(z)$ to $P$ moves $\phi(z)$ a distance at most $2\beta r$.

Thus, for all $z\in P \cap B'$,
\begin{equation}\label{eq:hclose}
     |h(z)-z| \leq (2L\epsilon+2\beta)r \leq (2L\epsilon + 2\delta)r < 0.5r.
\end{equation}
In particular, $h(P\cap B')\subseteq P\cap 2B$. Moreover, $h$ moves each point of $P\cap B'$ a distance less than $0.5r$, and $B'$ is a ball centered in $P$ of radius $1.5 r$.

Basic topological degree theory then implies that 
\begin{equation}\label{eq:hsurjective}
 h(P\cap B') \supseteq P \cap B.
\end{equation}
(For instance, given $y\in P\cap B$ one may apply Brouwer's fixed point theorem to the map $z\mapsto x-h(x)+y$. This function maps $P\cap B'$ into itself by \eqref{eq:hclose}. A fixed point of this map is then a pre-image of $y$ under $h$.)

Now consider any point $z\in P\cap B$. Let $z'$ be a point in $P \cap B^n(x,(1-2\beta)r)$ with $|z'-z|\leq 2\beta r$. By \eqref{eq:hsurjective}, the point $z'$ is equal to $h(w)=p(\phi(w))$ for some $w\in P\cap B'$. 

Let $y=\phi(w)\in M\cap 2B$. This point is in $N_{2\beta r}(P)$, by \eqref{eq:MnearP}, so
$$ |y-z'| = |y-p(y)|\leq 2 \beta r.$$
In addition, since $(y-z') \perp (z'-x)$, we have
$$ |y-x| = \sqrt{|y-z'|^2 +|z'-x|^2} \leq \sqrt{(2\beta r)^2 + ((1-2\beta) r)^2} < r.$$
Thus, $y\in M \cap B$ and
$$ |z-y| \leq |z-z'| + |z'-y| \leq 2\beta r + 2\beta r.$$
This shows that $P\cap B \subseteq N_{4\beta r}(M\cap B)$, and hence that $\theta \leq 4\beta$.
\end{proof}

\section{Flatness of quasiplanes}\label{sec:quasiplanes}
Here we show how Corollary \ref{lem:ABT2} follows from Theorem \ref{thm:LLCfine}. Denote by $\textbf{0}$ the origin in $\R^N$.

\begin{proof}[{Proof of Corollary \ref{lem:ABT2}}]
Applying appropriate similarities and using Remark \ref{rem:similarities}, we may assume that $v={\bf 0}$, $f({\bf 0})={\bf 0}$, $r=1$, and $\diam{f(B^N({\bf 0},2))} =1$. We may further assume that $H_f(B^N({\bf 0},2))<2$, as otherwise there is nothing to prove. In this case, $f$ is weakly $H$-quasisymmetric on $B^N({\bf 0},2)$ for some $H\in (1,2)$, thus also weakly $2$-quasisymmetric. 

Fix $e\in \mathbb{S}^{N-1}$ and let $R=\frac12|f(e)|$. By the weak 2-quasisymmetry of $f$, we have that $\frac12|f(e)| \leq |f(x)|$ for every $x\in \mathbb{S}^{N-1}$. Therefore, $B^N({\bf 0},R) \subset f(B^N({\bf 0},2))$.

By Lemma \ref{lem:weaktoQS}, there exists a homeomorphism $\eta_0:[0,\infty) \to [0,\infty)$ that depends only on $N$, such that $f$ is $\eta_0$-quasisymmetric. By the invariance of the LLC property under quasisymmetric maps (see \cite[p.137]{BonK}), there exists $C_0>1$ that depends only on $N$ such that $f(B^N({\bf 0},2))$ and $f(\partial B^N({\bf 0},2))$ are $(C_0,1)$-LLC. 

For the rest of the proof, abusing notation, we identify $\R^N$ with the plane $\R^N\times\{0\}$ in $\R^{N+1}$. Define $\Omega = f(B^N({\bf 0},2)) \times [0,3] \subset \R^{N+1}$. Note that $\partial\Omega$ is a topological $N$-sphere with $f(B^N({\bf 0},2))\subset \partial\Omega$ and if $\rho\leq R$, then
\[ B^{N+1}({\bf 0},\rho) \cap \partial\Omega = B^{N}({\bf 0},\rho) \cap f(B^N({\bf 0},2)).\] 
We claim that $\partial\Omega$ is $(C_1,1)$-LLC for some $C_1>1$ depending only on $N$. For the proof of the claim, fix $x\in \partial\Omega$ and let $\rho\in (0,1)$. Note that $B^{N+1}(x,\rho)$ intersects at most one of the sets $f(B^N({\bf 0},2))$ and $f(B^N({\bf 0},2))\times\{3\}$. There are 3 possible cases. 

\emph{Case 1.} Suppose that $B^{N+1}(x,r)$ intersects $f(B^N({\bf 0},2))$. Denote by $\pi_1:\R^{N+1}\to\R^N\times\{0\}$ the orthogonal projection. Then $|x-\pi_1(x)|\leq \rho$, 
\[ \pi_1(B^{N+1}(x,\rho)\cap\partial\Omega) = B^N(\pi_1(x),\rho) \subset f(B^N({\bf 0},2)),\] 
and we can linearly homotope $B^{N+1}(x,\rho)\cap\partial\Omega$ onto $B^N(\pi_1(x),\rho)\cap f(B^N({\bf 0},2))$ inside $\partial\Omega \cap B^{N+1}(x,2\rho)$. Now, using the $(C_0,1)$-LLC property of sets $f(B^N({\bf 0},2))$, we can contract the set $B^N(\pi_1(x),\rho)\cap f(B^N({\bf 0},2))$ to a point inside 
\[ B^N(\pi_1(x),C_0\rho)\cap f(B^N({\bf 0},2)).\]

\emph{Case 2.} Suppose that $B^{N+1}(x,r)$ intersects $f(B^N({\bf 0},2))\times\{3\}$. We work as in Case 1, using a projection $\pi_2$ on $\R^N\times\{3\}$.

\emph{Case 3.} Suppose that $B^{N+1}(x,r)\cap\partial\Omega$ intersects neither $f(B^N({\bf 0},2))\times\{0\}$ nor $f(B^N({\bf 0},2))\times\{3\}$. Let $x_{N+1}$ be the $(N+1)$-coordinate of $x$ and let $\pi_3$ be the projection of $\R^{N+1}$ onto $\R^N\times\{x_{N+1}\}$. We linearly homotope $B^{N+1}(x,r)\cap\partial\Omega$ onto its image under $\pi_3$, namely
\[\left(B^N(\pi_3(x),\rho)\cap (f(\partial B^N({\bf 0},2)))\right)\times\{x_{N+1}\},\] 
inside $B^{N+1}(x,r)\cap\partial\Omega$. Now, use the $(C_0,1)$-LLC property of $f(\partial B^N({\bf 0},2))$ to contract $B^N(\pi_3(x),\rho)\cap (f(\partial B^N({\bf 0},2)))\times\{x_{N+1}\}$ to a point inside 
\[ \left(f(\partial B^N({\bf 0},2))\times\{x_{N+1}\}\right) \cap B^{N+1}(x,C_0\rho). \]
This completes the proof of the claim.

Using our claim, Theorem \ref{thm:LLCfine} yields a constant $K>1$ that depends only on $N$ such that if $V\subset \R^N$ is an $n$-dimensional plane, then
\begin{align*} 
\theta_{f(V\cap B^N({\bf 0},2))}^n(B^N({\bf 0},\tfrac12R)) = \theta^n_{\partial\Omega}(B^N({\bf 0},\tfrac12R)) &\leq K\beta^n_{\partial\Omega}(B^N({\bf 0},R))\\
&= K \beta_{f(V\cap B^N({\bf 0},2))}^n(B^N({\bf 0},R))\\
&\leq 72NK(H_f(B^N({\bf 0},2)) -1) \end{align*}
with the last inequality following from Lemma \ref{lem:ABT}. This completes the proof of the corollary.
\end{proof}

\section{Menger curvature and parametrizations}\label{sec:mengerparam}

In this section we relate the Menger curvature with local smooth parametrizations and prove Theorem \ref{cor:kolasinski}. 
%

We start by generalizing the notion of geometric curvature energy. 
Fix a set $X\subset \mathbb{R}^n$ and a parameter $m\in \{0,1, \dots, n\}$. Following \cite{K}, for $l\in \{1, 2, \dots, m+2\}$, and $p>0$, we define the geometric curvature energy
$$ \mathcal{E}^{m,l}_p(X) = \int_{X^l} \sup_{x_l, \dots, x_{m+1}\in X} \mathcal{K}(x_0, \dots, x_{m+1})^p \, d\mathcal{H}^{ml}(x_0,\dots,x_{l-1}).$$
Note that, as a special case, $\mathcal{E}^{m,m+2}_p(X) = \mathcal{E}^{m+2}_p(X)$.

We also recall two important properties of Menger curvature that will be used in the next section. The first is a scaling property.

\begin{lemma}\label{lem:mengerscale}
Let $X\subset \R^N$ be a compact set and let $m\in\{0,\dots,N\}$ and $p>0$. If $\phi:\R^N \to \R^N$ is a similarity map with scaling factor $\lambda>0$, then
\[ \mathcal{E}^{m,l}_p(\phi(X)) = \lambda^{ml-p}\mathcal{E}^{m,l}_p(X). \]
\end{lemma}

\begin{proof}
Given $y_0,\dots,y_{l-1} \in \phi(X)$, let $x_i = \phi^{-1}(y_i)\in X$ for $i=0,\dots,l-1$, and by \eqref{eq:Hausscale} applied on $\mathcal{H}^{m+1}$,
\begin{align*}
\sup_{y_l,\dots,y_{m+1} \in \phi(X)}&\left( \frac{\mathcal{H}^{m+1}(\Delta(y_0,\dots,y_{m+1}))}{(\diam{\Delta(y_0,\dots,y_{m+1})})^{m+2}} \right)^p\\
&= \sup_{x_l,\dots,x_{m+1} \in X}\left( \frac{\mathcal{H}^{m+1}(\Delta(\phi(x_0),\dots,\phi(x_{m+1})))}{(\diam{\Delta(\phi(x_0),\dots,\phi(x_{m+1}))})^{m+2}} \right)^p\\ 
&= \lambda^{-p} \sup_{x_l,\dots,x_{m+1} \in X}\left( \frac{\mathcal{H}^{m+1}(\Delta(x_0,\dots,x_{m+1}))}{(\diam{\Delta(x_0,\dots,x_{m+1})})^{m+2}} \right)^p.
\end{align*}
Now applying \eqref{eq:Hausscale} on $\mathcal{H}^{ml}$, we obtain the lemma.
\end{proof}

The second property is about the finiteness of Menger energies on $C^{1,\alpha}$ embedded compact manifolds. Kolasi\'nski \cite[Proposition 1.9]{K} showed that If $M\subset\R^n$ is a compact, $m$-dimensional $C^2$ manifold embedded in $\R^n$, then the discrete curvature $\mathcal{K}$ is bounded on $M^{m+2}$. While this is enough for our main theorems, we give below a stronger version of this result with $C^2$ replaced by $C^{1,1}$.

\begin{lemma}\label{lem:C2}
Let $M\subset\R^n$ be a compact, $m$-dimensional $C^{1,\alpha}$ manifold embedded in $\R^n$. If $\alpha <1$, then $\mathcal{K}$ is in $L^p(M^{m+2})$ for all $p< \frac{m(m+1)+1}{1-\alpha}$. If $\alpha=1$, then $\mathcal{K}$ is bounded on $M^{m+2}$.
\end{lemma}

\begin{proof}
We first recall a well-known fact from differential geometry  (as a corollary of the Inverse Function Theorem): every $C^{1,\alpha}$ $m$-submanifold of $\R^n$ is locally the graph of a $C^{1,\alpha}$ function. More precisely, given $x\in M$, there exists an isometry $\Phi:\R^n\to \R^n$, a neighborhood $U$ of $\textbf{0}$ in $\R^m$, a radius $R>0$, and functions $f_i:U \to \R$, $i=1,\dots,n-m$ such that
\begin{enumerate}
\item $\Phi(x_0) = \textbf{0}$ and the tangent $m$-plane to $\Phi(M)$ at $\textbf{0}$ is $\R^{m}\times\{0\}^{n-m}$,
\item the function $f: U \to \R^n$ with 
\[ f(t_1,\dots,t_m) = (t_1,\dots,t_m,f_1(t_1,\dots,t_m),\dots,f_{n-m}(t_1,\dots,t_m))\]
maps $U$ onto $B^n(\textbf{0},R)\cap \Phi(M)$,
\item for each $i\in\{1,\dots,n-m\}$, $f_i(\textbf{0})=0$ and $f_i$ has partial derivatives that are $\alpha$-H\"older continuous with some H\"older constant $H$.
\end{enumerate}
It is easy to see that points near $x_0$ have the same property with H\"older constant comparable to $H$. Therefore, by compactness of $M$, we may assume that the constant $H$ as well as the radius $R$ are independent of the point $x_0\in M$.

The latter now implies that all points of $M$ have uniformly small betas. Fix $x\in M$ and $r\in (0,R)$. If $P$ is the $m$-plane in $\R^n$ that is tangent to $M$ at $x$, then for all $y\in M\cap B^n(x,r)$,
\[ \dist(y,P) \leq \sqrt{m}(n-m)H|x-y|^{1+\alpha}\]
which implies that $\beta_M^m(B(x,r)) \lesssim r^{\alpha}$. 

By \cite[Lemma 1.10]{K} we have that if $x_1,\dots,x_{m+2}$ are points on $M$ with $\diam{\Delta(x_1,\dots,x_{m+2})} = d$, then $\mathcal{K}(x_,\dots,x_{m+2}) \lesssim d^{\alpha-1}$. This implies now that if $\alpha=1$, then $\mathcal{K}$ is bounded on $M^{m+2}$.

To finish the proof, assume that $\alpha<1$. Assume also, as we may, that $\diam{M}=1$. Fix $p \in [1,\frac{m(m+1)+1}{1-\alpha})$ and for each $k\in\N$ let 
\[ \mathscr{A}_k = \{(x_1,\dots,x_{m+2})\in M^{m+2} : \diam{\Delta(x_1,\dots,x_{m+2})} \in (2^{-k},2^{1-k}]\}. \]
We have
\begin{align*}
\mathcal{E}_p^{m+2}(M) &= \int_{M^{m+2}}\mathcal{K}(x_1,\dots,x_{m+2})^p\, d\mathcal{H}^{m(m+2)}(x_1,\dots,x_{m+2})\\
&= \sum_{k\in\N} \int_{\mathscr{A}_k}\mathcal{K}(x_1,\dots,x_{m+2})^p\, d\mathcal{H}^{m(m+2)}(x_1,\dots,x_{m+2})\\
&\lesssim \sum_{k\in\N} 2^{kp(1-\alpha)}\mathcal{H}^{m(m+2)}(\mathscr{A}_k).
\end{align*}
Note that if $(x_1,\dots,x_{m+2}) \in \mathscr{A}_k$, then all $x_i$ are in  $B^n(x_1,2^{1-k})$. Therefore, 
\[\mathscr{A}_k \subset \bigcup_{x\in M}\{x\} \times (M\cap B(x,2^{1-k}))^{m+1}\] 
and
\begin{align*}
&\mathcal{E}_p^{m+2}(M)\\ 
&\lesssim  \sum_{k\in\N} 2^{kp(1-\alpha)} \int_{M} \int_{(B^n(x,2^{1-k})\cap M)^{m+1}}\, d\mathcal{H}^{m(m+1)}(x_2,\dots,x_{m+2}) \, d\mathcal{H}^m(x)\\
&\lesssim \sum_{k\in\N} 2^{kp(1-\alpha)} 2^{-km(m+1)}\\
&= \sum_{k\in\N} 2^{k\left(p(1-\alpha)-m(m+1)\right)}\\
&<\infty. \qedhere
\end{align*}

%
%

\end{proof}

We complete this subsection with the following more general version of Theorem \ref{thm:kolasinski2} and its proof.

\begin{theorem}\label{cor:kolasinski}
Let $M$ be a compact LLC $m$-manifold in $\RR^n$. Suppose that 
$$ \mathcal{E}^{m,l}_p(M) < \infty$$
for some $p>ml$.  Then $M$ is a $C^{1,\alpha}$ manifold with $\alpha=\frac{p-ml}{(p+ml)(m+1)}$.
\end{theorem}

To prove Theorem \ref{cor:kolasinski}, we would like to simply apply \cite[Theorem 1]{K}. Unfortunately, the assumptions of that theorem do not quite match up with the conclusion of Theorem \ref{thm:LLCfine}, primarily because of the factor of $2$ multiplying the radius of the ball on the right hand side of \eqref{eq:fine}. However, the arguments are essentially unchanged by this factor, as we now sketch.

\begin{proof}[Proof of Theorem \ref{cor:kolasinski}]
Let $M$ be a compact LLC $m$-manifold in $\mathbb{R}^n$, and assume that $\mathcal{E}^{m,l}_p(M) \leq E <\infty$ for some $l\in\{1, \dots, m+2\}$ and $p>ml$. We first observe that $M$ satisfies the following lower Ahlfors regularity bound: There is a constant $C>0$ such that
$$ \mathcal{H}^m(M\cap B^n(x,r)) \geq Cr^m, \quad\text{for all } x\in M, \, 0<r\leq \diam(M).$$
This follows from a result of Kinneberg \cite[Corollary 1.4]{Kinneberg}. 

The rest of the argument proceeds as in the proof of \cite[Theorem 1]{K}. First, \cite[Proposition 1]{K} shows that there exists $C_0 = C_0(E,C,m,l,p)$ such that 
\begin{equation}\label{eq:betadecay}
 \beta_M^m(B^n(x,r)) \leq C_0 r^{\alpha},\quad\text{for all } x\in M, \, 0<r\leq \diam(M),
 \end{equation}
where $\alpha = \frac{p-ml}{(p+ml)(m+1)}$. By Theorem \ref{thm:LLCfine}, it follows that $M$ is Reifenberg flat with vanishing constant. Applying a result of David--Kenig--Toro (see \cite[Proposition 9.1]{DKT} or \cite[Proposition 1.3]{K}), we obtain that $M$ is a $C^{1,\alpha}$ submanifold.
\end{proof}

\section{Proof of Theorem \ref{thm:ex}}\label{sec:ex}


We use the following symbolic notation. For each integer $m\geq 0$ let $\{1,2\}^m$ be all words formed from $\{1,2\}$ of length $m$ with the convention $\{1,2\}^0 = \{\varepsilon\}$ where $\varepsilon$ denotes the empty word. We also set $\{1,2\}^*=\bigcup_{m\geq 0}\{1,2\}^m$. If $w\in \{1,2\}^*$, then we set $|w|$ to be the number of letters that $w$ has with the understanding that $|\varepsilon|=0$. If $w\in\{1,2\}^* \setminus\{\varepsilon\}$ then there exists unique $u\in\{1,2\}^*$ and $i\in 1,2$ such that $w=ui$; we denote this $u$ by $w^{\uparrow}$ and we call it the parent of $w$. Finally, we denote by $\{1,2\}^{\N}$ the set of infinite words. Naturally, if $w\in\{1,2\}^{\N}$, we define $|w| = \infty$.

Here and for the rest of this section, given two integers $1\leq k< N$, we conflate $\R^{k}$ with the $k$-plane $\R^k \times \{0\}^{N-k}$ in $\R^{N}$. Similarly, sets $A\subset \R^k$ are conflated with sets $A\times\{0\}^{N-k}$ in $\R^{N}$.

In \textsection\ref{subsec:constr} we construct the spheres $\Sigma_n$ and in \textsection\ref{subsec:estimates} we show that the energies $\mathcal{E}_p^{n+2}(\Sigma_n)$ are finite.

\subsection{Construction of $\Sigma_n$}\label{subsec:constr}
Fix $n\geq 3$. The construction of $\Sigma_n$ is essentially given in \cite[Section 2]{PV} but we recall it here.

Let $\psi: (\S^1)^{n-2} \to (\S^1)^{n-2}$ be the cyclic permutation
\[ \psi(x_1,x_2,\dots,x_{n-2}) = (x_{n-2},x_1,\dots,x_{n-3}) \]
with the understanding that $\psi$ is the identity on $\S^1$ when $n=3$.

Let $F:\mathbb{B}^2 \times\S^1 \to \R^3$ be a smooth embedding and let $t_{\varepsilon}$ be the image of $F$. For $i\in \{1,2\}$, let $t_1,t_2$ be two smooth solid tori (i.e., diffeomorphic to $t_{\varepsilon}$) contained in the interior of $t_{\varepsilon}$ linked with each other, forming the Bing's double \cite{Bing}; see Figure \ref{fig:bing}. Let also $\varphi_i : t_{\varepsilon} \to t_{i}$ for $i=1,2$ be two diffeomorphisms. For $i=1,2$ let 
\[ \phi_i : \mathbb{B}^2 \times\S^1 \to \mathbb{B}^2 \times\S^1 \quad\text{by}\quad \phi_i = F^{-1}\circ \varphi_i\circ F \] 
and
\[\Phi_i : \mathbb{B}^2\times (\S^1)^{n-2} \to \mathbb{B}^2\times (\S^1)^{n-2}\qquad\text{by}\quad \Phi_i = (\phi_i\times \text{id}_{(\S^1)^{n-3}})\circ(\text{id}_{\mathbb{B}^2}\times\psi).\] 
\begin{figure}
\includegraphics[width=0.5\textwidth]{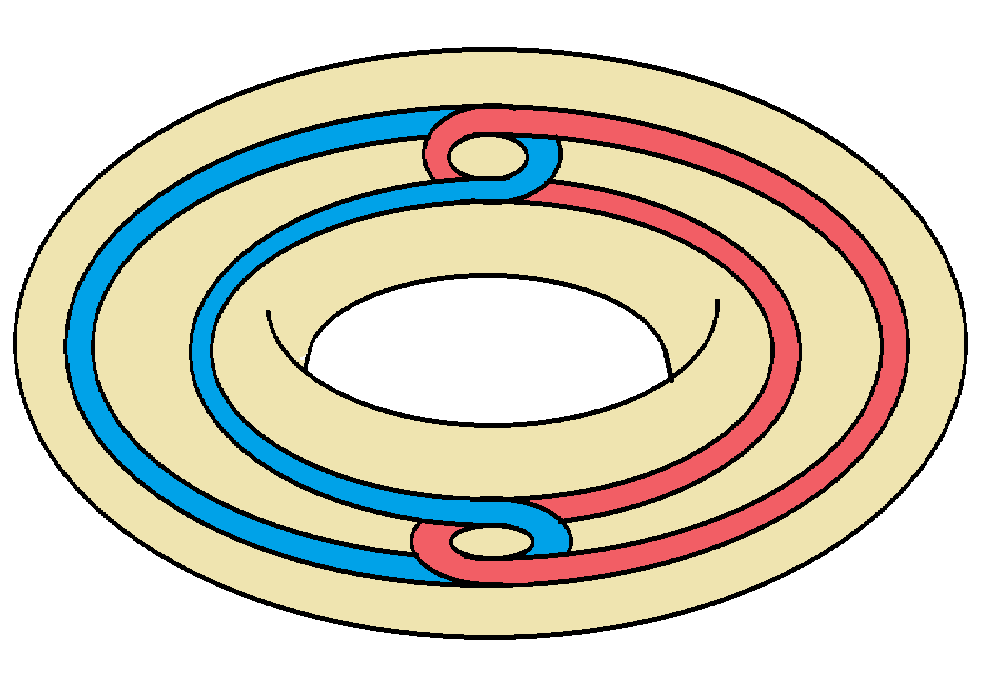}
\caption{The solid tori $t_1,t_2$ (in blue and red) linked inside $t_{\varepsilon}$.}\label{fig:bing}
\end{figure}
Define
\[ \T_{\varepsilon} = \mathbb{B}^2\times (\S^1)^{n-2}\]
and for $w = i_1\cdots i_k\in \{1,2\}^k$ define
\[ \Phi_w = \Phi_{i_1}\circ\cdots\circ\Phi_{i_k}\qquad\text{and}\qquad \mathbb{T}_w = \Phi_w(\mathbb{T}_{\varepsilon}).\]

Following \cite{Blankinship}, there exists a smooth embedding 
\[ \vartheta : \T_{\varepsilon} \to \mathbb{B}^n \subset \R^{n}\] 
for which there exists $x_0 \in \S^1$ such that
\begin{enumerate}
\item $\vartheta(\mathbb{B}^2\times\S^1\times\{x_0\}^{n-3}) \subset \R^3\times\{0\}^{n-3}$,
\item $\vartheta\left( \phi_i(\mathbb{B}^2\times\S^1) \times\{x_0\}^{n-3} \right) \subset \R^3\times\{0\}^{n-3}$ for $i\in \{1,2\}$,
\item $\vartheta\circ(\text{id}_{\mathbb{B}^2}\times\psi^k) = (\text{id}_{\R^2}\times\psi')^k \circ \vartheta$ for all $k\in\N$ where $\psi' : \R^{n-2}\to\R^{n-2}$ is the cyclic permutation
\[ \psi'(x_1,x_2,\dots,x_{n-2}) = (x_{n-2},x_1,\dots,x_{n-3}). \]
\end{enumerate} 
For each $w\in \{1,2\}^*$, let $T_w = \vartheta(\T_w)$. For each $i\in \{1,2\}$, let $\tilde{\Phi}_i : T_{\varepsilon}\to T_{i}$ be the diffeomorphism given by 
\[ \tilde{\Phi}_i = \vartheta\circ \Phi_i \circ \vartheta^{-1}.\]
If $w=i_1\cdots i_k \in \{1,2\}^k$ and $k\in\N$, then
\[ \tilde{\Phi}_w := \tilde{\Phi}_{i_1}\circ \cdots \circ \tilde{\Phi}_{i_k} : T_{\varepsilon} \to T_w\]
is a diffeomorphism that maps $T_{\varepsilon}$ onto $T_w$.

Set $B= B^n({\bf 0},5) \subset \R^{n}$ and recall that $T_{\varepsilon}\subset B$ and that $\dist(T_{\varepsilon},\R^n\setminus B)\geq 4$. Set 
\[ U:=B\times [-4,4] \subset \R^{n+1}\] 
and define the cylinders
\[ C_w = T_w \times [-2^{-|w|},2^{-|w|}] \subset \R^{n+1}, \qquad \text{for $w\in \{1,2\}^*$}.\]

For $i\in \{1,2\}$, let
\[g_1(x) = (1/2)x + b_1, \quad g_2(x)=(1/2)x + b_2\]
be two similarities of $\R^{n+1}$ such that $g_i(T_{\varepsilon}) \subset\R^n\times\{\frac12\}$ and the images $g_1(T)$ and $g_2(T)$ are pairwise disjoint. If $w=i_1\cdots i_k \in \{1,2\}^k$, then we denote
\[ g_w = g_{i_1}\circ \cdots \circ g_{i_k}.\]

\begin{lemma}\label{lem:Gk}
There exists a family of diffeomorphisms $\{G_k : U \to U\}_{k\in\N}$ such that for each $k\in\N$,
\begin{enumerate}[(a)]
\item $G_k$ is the identity on a neighborhood of $\partial T_{\varepsilon}$ and on $U\setminus T_{\varepsilon}$,
\item 
\[ G_{k+1} | \left(U \setminus \bigcup_{w\in \{1,2\}^{k}}C_w\right)  = G_{k} | \left(U \setminus \bigcup_{w\in \{1,2\}^{k}}C_w\right),  \]
\item if $w\in \{1,2\}^k$, then $G_{k}(C_w) = g_w(C_{\varepsilon})$,

\item if $w \in \{1,2\}^*$ with $|w| \leq k-1$, then $G_k(T_w \setminus (T_{w1}\cup T_{w2}))$ is a rescaled copy of $G_1(T_{\varepsilon} \setminus (T_{1}\cup T_{2}))$ by a factor of $2^{-|w|}$.
\end{enumerate}
\end{lemma}

\begin{proof}
We start by constructing $G_1$. We give only a sketch of the construction; the ideas are similar to those in \cite[Proposition 3.1]{PV}. Recall that we have assumed that there exists $x_0 \in \S^1$ for which
\[ \tau_{\varepsilon}:= \vartheta(\mathbb{B}^2\times\S^1\times\{x_0\}^{n-3})\]
is a $3$-tube in $\R^3$ with smooth boundary and
\[ \tau_i := \vartheta \circ \tilde{\Phi}_i(\mathbb{B}^2\times\S^1\times\{x_0\}^{n-3}), \qquad\text{for $i=1,2$} \]
are $3$-tubes in $\tau_{\varepsilon}$ with smooth boundary linked to each other as in Figure \ref{fig:bing}. Let now
\[ \theta = (\theta_1,\theta_2,\theta_3,\theta_4) : C_{\varepsilon} \cap \R^3 \to \R^4\]
be Semmes' reembedding (see \cite[Definition 3.2]{Semmes1996}) which has the following properties:
\begin{enumerate}
\item there exists a neighborhood $V$ of $\partial C_{\varepsilon}$ such that $\theta$ is the identity on $V\cap C_{\varepsilon}\cap\R^3$,
\item $\theta$ unlinks $\tau_1$ and $\tau_2$ in $\R^4$ and maps them onto $g_1(\tau_{\e})$ and $g_2(\tau_{\e})$, respectively,
\item $\theta(C_{\varepsilon}\cap \R^{3}) \subset \R^3\times[0,\frac12]$. 
\end{enumerate} 
The diffeomorphism $G_1$ is now obtained by extending the embedding 
\begin{equation}\label{eq:G1}
C_{\varepsilon}\cap \R^3 \to \R^{n+1}, \qquad x\mapsto (\theta_1(x),\theta_2(x),\theta_3(x),x_4,\dots,x_n,\theta_4(x)) 
\end{equation}
using Semmes' isotopy extension lemma \cite[Lemma 4.1]{Semmes1996}. Property (a) follows from (1) and \eqref{eq:G1}, while (c) follows from (2) and \eqref{eq:G1}. Properties (b) and (d) are vacuous.

We proceed iteratively. Assume that for some $k\in\N$ we have defined a diffeomorphism $G_{k}: U \to U$ satisfying properties (a)--(d) of the lemma. By the induction hypothesis, if $w\in \{1,2\}^k$, then $g_w(C_{\varepsilon})=G_k(C_w)$. To obtain $G_{k+1}$, we change the definition of $G_k$ on each $C_w$ (with $w\in\{1,2\}^k$) replacing $G_k(C_w)$ by $g_w(G_1(C_{\varepsilon}))$. More specifically, define $G_{k+1}: U \to U$ by
\[ G_{k+1} | \left(U \setminus \bigcup_{w\in \{1,2\}^{k}}C_w\right)  = G_{k} | \left(U \setminus \bigcup_{w\in \{1,2\}^{k}}C_w\right)\] 
and for $w\in\{1,2\}^k$ we define $G_{k+1}$ on $C_w$ by
\begin{equation}\label{eq:Gk+1}
G_{k+1}(x,t) =  g_w\circ G_1(\tilde{\Phi}_w^{-1}(x),2^kt), \quad\text{for $x\in T_{w}$ and $t\in [-2^{-k},2^{-k}]$}. 
\end{equation} 
Recall that the map if $w\in\{1,2\}^k$, then the map $(x,t) \mapsto (\tilde{\Phi}_w^{-1}(x),2^kt)$ is a diffeomorphism of $C_w$ onto $C_{\varepsilon}$.

Property (b) is clear by design of $G_{k+1}$ while property (a) follows from property (b) and the inductive hypothesis. Property (c) follows from \eqref{eq:Gk+1}. Fix now $w \in \{1,2\}^*$ with $|w| \leq k$. If $|w| \leq k-1$, then we get (d) from (b) and the inductive hypothesis. If $|w|=k$, then we get (d) from (c). 
\end{proof}

By Lemma \ref{lem:Gk}, if $w\in \{1,2\}^*$ and $k,n \in\N$ with $\min\{k,n\}\geq |w|+1$, then
\[ G_k(T_w\setminus (T_{w1}\cup T_{w2})) = G_n(T_w\setminus (T_{w1}\cup T_{w2})). \]

Denote by ${\bf 0}$ the origin in $\R^{n+1}$ and for each $m\in\N$ let $P_m=(\tfrac32 2^{-m-1},0,\dots,0)\in\R^{n+1}$ and 
\begin{align*}
B_m &= B^{n+1}({\bf 0}, 2^{-m}), \\
B_m^* &= B^{n+1}({\bf 0},\tfrac{15}{16}2^{-m}),\\
B_m' &= B^{n+1}(P_m, 2^{-m-3}), \\
B_m'' &= B^{n+1}(P_m, 2^{-m-4}),\\
\tilde{B}_m &= B^{n+1}(P_m, 2^{-m^2-5}).
\end{align*}
Also define for each $m\in\N$ the similarity
\[ \zeta_m : \R^{n+1} \to\R^{n+1} \quad \text{by}\quad \zeta_{m}(x) = \frac{2^{-m^2-5}}{10}x + P_m,\]
which maps the $n$-dimensional unit ball $\mathbb{B}^n$ onto $\tilde{B}_m\cap (\R^n\times\{0\})$.

\begin{figure}
    \centering
    \includegraphics[scale=0.6]{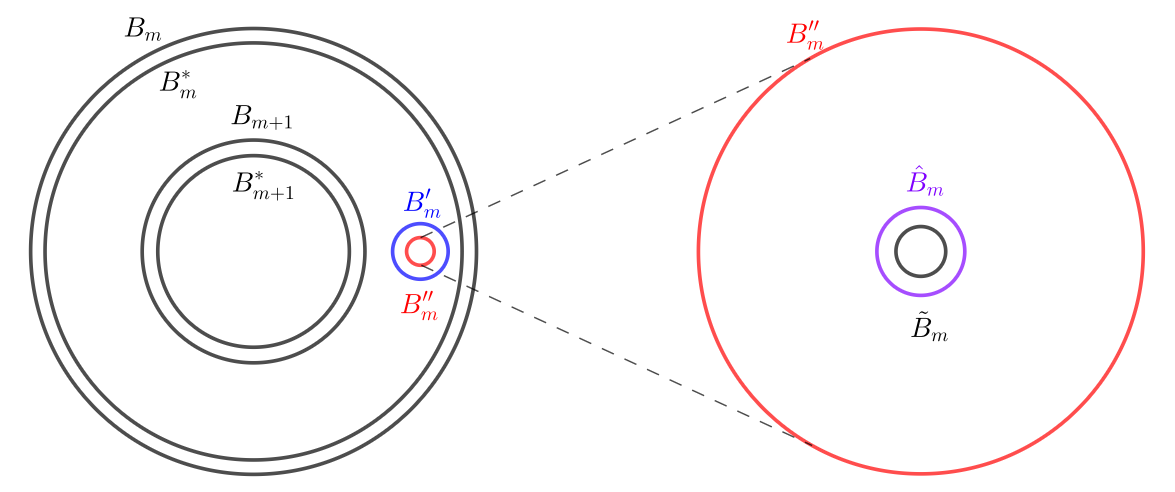}
    \caption{Balls $B_m$, $B_m^*$, $B_m'$, $B_m''$, $\tilde{B}_m$, and $\hat{B}_m$.}
\end{figure}

Define
\[ \Sigma_n' =  \left( B \setminus \bigcup_{m=1}^{\infty}B_m^*\right) \cup \left(\bigcup_{m=1}^{\infty}\zeta_m\circ G_m(B)\right).\]

By \cite[Proposition 3.2]{PV}, $\Sigma_n'$ is homeomorphic to the ball $\mathbb{B}^n$ and smooth outside of the origin ${\bf 0}$. By \cite[Corollary 3.6]{PV}, $\Sigma_n'$ is LLC and Ahlfors $n$-regular. Gluing $\Sigma_n'$ to a sphere without a disk in a smooth fashion, we may assume for the rest that $\Sigma_n'$ is contained in an $n$-sphere $\Sigma_n$ which is LLC, $n$-regular and smooth outside of ${\bf 0}$. By \cite[\textsection7.2]{PV}, $\Sigma_n'$ does not admit a quasisymmetric parametrization by the unit ball $\mathbb{B}^n$ so $\Sigma_n$ is not quasisymmetrically equivalent to $\S^n$. 



\subsection{Energy estimates}\label{subsec:estimates}

To finish the proof of Theorem \ref{thm:ex} it remains to establish the next proposition.

\begin{proposition}\label{prop:estimates}
For each $p\in (1,n(n+2))$ we have $\mathcal{E}_p^{n+2}(\Sigma_n)< \infty$.
\end{proposition}

The proof of Proposition \ref{prop:estimates} occupies the rest of this section. To this end, fix $p\in (1,n(n+2))$. To simplify the notation, we write $\Sigma=\Sigma_n$. Below, given positive quantities $A,B>0$, we write $A \lesssim B$ if there exists a constant $C\geq 1$ that depends at most on $n$, $p$, and the Ahlfors regularity constant for $\mathcal{H}^n$, such that $A\leq C B$. We also write $A\simeq B$ is $A \lesssim B$ and $A\gtrsim B$. 

We start by splitting the set $\Sigma^{n+2}$ into several subsets. Let
\begin{align*}
\mathscr{A}_0 &:= (\Sigma\setminus (B_2 \cup B_1''))^{n+2}\\
\mathscr{B}_0 &:= \{ (x_1,\dots,x_{n+2}) \in \Sigma^{n+2} : \text{for some distinct $i,j$},\\
&\qquad\qquad\qquad \text{$x_i \in B_2\cup B_1''$ and $x_j \in \Sigma\setminus B_1$}\}
\end{align*}
and for each $m\in\N$ let
\begin{align*}
\mathscr{A}_m &:= ((\Sigma\cap B_m)\setminus (B_{m+1}^*\cup B_m''))^{n+2}\\
\mathscr{B}_m &:= \{ (x_1,\dots,x_{n+2}) \in (\Sigma\cap B_m)^{n+2} : \text{for some distinct $i,j$},\\
&\qquad\qquad \text{$x_i \in B_{m+1}^* $ and $x_j \in B_m\setminus B_{m+1}$}\}\\
\mathscr{C}_m &:= \{ (x_1,\dots,x_{n+2}) \in (\Sigma\cap B_m)^{n+2} : \text{for some distinct $i,j$},\\
&\qquad\qquad \text{$x_i \in B_{m}''$ and $x_j \in B_m\setminus B_{m}'$}\}
\end{align*}

\begin{lemma}
We have that 
\begin{equation}\label{eq:partition}
\Sigma^{n+2} = \{{\bf 0}\}^{n+2} \cup (\mathscr{A}_0\cup \mathscr{B}_0) \cup \bigcup_{m\in\N}(\mathscr{A}_m\cup \mathscr{B}_m\cup \mathscr{C}_m)\cup \bigcup_{m\in\N}(\Sigma\cap B_m')^{n+2}.
\end{equation}
\end{lemma}

\begin{proof}
Fix $\textbf{x} =(x_1,\dots,x_{n+2}) \in \Sigma^{n+2}$. If $\textbf{x} \not\in \mathscr{A}_0$, then there exists $i$ such that $x_i \in B_2\cup B_1''$. There are now two possibilities. Either, there exists $j\neq i$ such that $x_j \in \Sigma\setminus B_1$ (in which case $\textbf{x}\in \mathscr{B}_0$), or $x_j \in B_1$ for all $j$ (in which case $\textbf{x}\in (\Sigma\cap B_1)^{n+2}$). This shows that
\begin{equation}\label{eq:partition1}
\Sigma^{n+2} = \mathscr{A}_0 \cup \mathscr{B}_0 \cup (\Sigma\cap B_1)^{n+2}.
\end{equation}

We now show that for all $m\in\N$ 
\begin{equation}\label{eq:partition2} 
(\Sigma\cap B_m)^{n+2} = \mathscr{A}_m \cup \mathscr{B}_m \cup \mathscr{C}_m \cup (\Sigma\cap B_m')^{n+2}\cup (\Sigma\cap B_{m+1})^{n+2}.
\end{equation}
Since $\{{\bf 0}\}^{n+2} = \bigcap_{m\in\N}(\Sigma\cap B_{m+1})^{n+2}$, \eqref{eq:partition1} and \eqref{eq:partition2} give \eqref{eq:partition}. 

Suppose that $\textbf{x} =(x_1,\dots,x_{n+2}) \in (\Sigma\cap B_m)^{n+2}$ for some $m\in\N$. Suppose also that $\textbf{x} \not\in \mathscr{A}_m$. Then there exists $i$ such that $x_i \in B_{m+1}^*\cup B_m''$. There are two cases to consider.

\emph{Case 1.} Suppose that $x_i \in B^*_{m+1}$. If $x_j \in B_{m+1}$ for all $j$, then $\textbf{x}\in (\Sigma\cap B_{m+1})^{n+1}$. On the other hand, if there exists $j\neq i$ such that $x_j \not\in B_{m+1}$, then $\textbf{x} \in \mathscr{B}_m$.

\emph{Case 2.} Suppose that $x_i \in B_{m}''$. If $x_j \in B_{m}'$ for all $j$, then $\textbf{x}\in (\Sigma\cap B_{m}')^{n+1} = \mathscr{C}_m$. On the other hand, if there exists $j\neq i$ such that $x_j \not\in B_{m}'$, then $\textbf{x} \in \mathscr{C}_m$.
\end{proof}

We now apply four estimates. First, $\Sigma\setminus (B_2 \cup B_1'')$ is a $C^2$-manifold, so by Lemma \ref{lem:C2} we have that 
\[ \int_{\mathscr{A}_0} \mathcal{K}(\textbf{x})^p\, d\mathcal{H}^{n(n+2)}(\textbf{x}) = \mathcal{E}_p^{n+2}(\Sigma\setminus (B_2 \cup B_1''))< \infty.\] 

Second, for all $m\in\N$,
\[ (\Sigma\cap B_m)\setminus (B_{m+1}^*\cup B_m'') \subset \R^n\times\{0\}\]
so 
\[ \int_{\mathscr{A}_m} \mathcal{K}(\textbf{x})^p\, d\mathcal{H}^{n(n+2)}(\textbf{x}) = 0.\] 

Third, if $\textbf{x} \in \mathscr{B}_0$, then $\diam{\Delta(\textbf{x})} \simeq 1$ which gives $\mathcal{K}(\textbf{x}) \lesssim (\diam{\Sigma})^{-1} \lesssim 1$.

Therefore, by the $n$-regularity of $\Sigma$ we get
\[ \int_{\mathscr{B}_0} \mathcal{K}(\textbf{x})^p\, d\mathcal{H}^{n(n+2)}(\textbf{x}) \lesssim 1. \]

Fourth, if $\textbf{x} \in \mathscr{B}_m\cup \mathscr{C}_m$, then $\diam{\Delta(\textbf{x})} \simeq 2^{-m}$, which gives $\mathcal{K}(\textbf{x}) \lesssim 2^{m}$. 
Therefore, by the $n$-regularity of $\Sigma$ we get
\begin{align*}
\int_{\mathscr{B}_m\cup\mathscr{C}_m} \mathcal{K}(\textbf{x})^p \, d\mathcal{H}^{n(n+2)}(\textbf{x})&\lesssim 2^{pm}\mathcal{H}^{n(n+2)}((\Sigma\cap B_m)^{n+2})\simeq 2^{-m(n(n+2)-p)}.
\end{align*}

Combining all these estimates, we have
\begin{align*}
\int_{\mathscr{A}_0\cup\mathscr{B}_0\cup\bigcup_{m\in\N}(\mathscr{A}_m\cup\mathscr{B}_m\cup\mathscr{C}_m)} \mathcal{K}(\textbf{x})^p \, d\mathcal{H}^{n(n+2)}(\textbf{x}) < \infty.
\end{align*}

Thus, to complete the proof of Proposition \ref{prop:estimates}, it remains to show that
\begin{equation}\label{eq:Cestimate}
\sum_{m\in\N}\int_{(\Sigma\cap B_m')^{n+2}} \mathcal{K}(\textbf{x})^p \, d\mathcal{H}^{n(n+2)}(\textbf{x}) < \infty.
\end{equation}
To do that, we further subdivide each $(\Sigma\cap B_m')^{n+2}$ as follows. For each $m\in\N$, let $\hat{B}_m = B^{n+1}(P_m,2^{-m^2-4})$ and let
\begin{align*}
\mathscr{D}_m &:= (\Sigma\cap (B_m'\setminus \hat{B}_m))^{n+2}\\
\mathscr{E}_m &:= \{(x_1,\dots,x_{n+2}) \in (\Sigma\cap B_m')^{n+2} : \text{for some distinct $i,j$,}\\
&\qquad\qquad\text{ $x_i \in B_m'\setminus \hat{B}_m$ and $x_j \in \tilde{B}_m$}\}.
\end{align*}

\begin{lemma}
For each $m\in\N$, $(\Sigma\cap B_m')^{n+2} = \mathscr{D}_m \cup \mathscr{E}_m \cup (\Sigma\cap \hat{B}_m)^{n+2}$.
\end{lemma}

\begin{proof}
Let $\textbf{x} = (x_1,\dots,x_{n+2})\in (\Sigma\cap B_m')^{n+2} \setminus (\Sigma\cap \hat{B}_m)^{n+2}$. Then, there exists $i$ such that $x_i \in \Sigma \cap(B_m' \setminus \hat{B}_m)$. There are two cases, either for all $j$ we have $x_j \in \Sigma \cap(B_m' \setminus \hat{B}_m) \subset \Sigma \cap(B_m' \setminus \tilde{B}_m)$ (which implies that $\textbf{x} \in \mathscr{D}_m$), or there exists $j\neq i$ such that $x_j \in \Sigma \cap \hat{B}_m$ (which implies that $\textbf{x} \in (\Sigma\cap \hat{B}_m)^{n+2}$).
\end{proof}

Since $\Sigma\cap (B_m'\setminus \hat{B}_m) \subset \R^n\times\{0\}$, we have that
\[ \int_{\mathscr{D}_m} \mathcal{K}(\textbf{x})^p \, d\mathcal{H}^{n(n+2)}(\textbf{x}) = 0.\]
If $\textbf{x} = (x_1,\dots,x_{n+2}) \in \mathscr{E}_m$, then set 
\[ M_{\textbf{x}} := \max_{i=1,\dots,n+2}|x_i-P_m|\] 
and note that $\mathcal{K}(\textbf{x}) \lesssim 1/M_{\textbf{x}}$. Then,
\begin{align*}
\sum_{m=1}^{\infty}\int_{\mathscr{E}_m} &\mathcal{K}(\textbf{x})^p \, d\mathcal{H}^{n(n+2)}(\textbf{x})\\
&= \sum_{m=1}^{\infty}\sum_{i=1}^{m^2-m+1}\int_{ M_{\textbf{x}} \in [2^{-m-3-i},2^{-m-2-i}]} \mathcal{K}(\textbf{x})^p \, d\mathcal{H}^{n(n+2)}(\textbf{x})\\
&\lesssim \sum_{m=1}^{\infty}\sum_{i=1}^{m^2-m+1} 2^{(m+i)p}2^{-n(n+2)(m+i)}\\
&\lesssim \sum_{m=1}^{\infty} 2^{-m(n(n+2)-p)}\\
&< \infty.
\end{align*}

Thus, to complete the proof, it remains to show that
\begin{equation}\label{eq:Cestimate2}
\sum_{m\in\N}\int_{(\Sigma\cap \hat{B}_m)^{n+2}} \mathcal{K}(\textbf{x})^p \, d\mathcal{H}^{n(n+2)}(\textbf{x}) < \infty.
\end{equation}
To this end, for $m\in\N$, $w\in \bigcup_{i=0}^{m-1}\{1,2\}^m$, and $u\in\{1,2\}^m$, write 
\begin{align*}
\Sigma_{m,0} &= \overline{\Sigma\cap \hat{B}_m} \setminus (\zeta_m\circ G_m(T_{\varepsilon})),\\
\Sigma_{m,w} &= \zeta_m \circ G_{m}(\overline{T_w\setminus (T_{w1}\cup T_{w2})}),\\
\Sigma_{m,u} &= \zeta_m \circ G_{m}(T_u).
\end{align*} and note that
\[ \overline{\Sigma\cap \hat{B}_m} = \Sigma_{m,0} \cup \bigcup_{i=0}^{m-1}\bigcup_{w\in\{1,2\}^i}\Sigma_{w,m} \cup \bigcup_{u\in\{1,2\}^m}\Sigma_{u,m}.\]

Given $x \in \zeta_m\circ G_m(T_{\varepsilon})$ denote by $\textbf{w}(x)$ the word $w\in\bigcup_{i=0}^m\{1,2\}^i$ of longest length such that $x\in \Sigma_{m,w}$. Equivalently, $\textbf{w}(x)=w$ if and only if
\begin{enumerate}
\item either $|w|=m$ and $x\in \Sigma_{m,w}$,
\item or $|w| \leq m-1$, and $x\in \Sigma_{m,w} \setminus(\Sigma_{m,w1}\cup \Sigma_{m,w2})$.
\end{enumerate}
Given $(x_1,\dots,x_{n+2}) \in (\zeta_m\circ G_m(T_{\varepsilon}))^{n+2}$ write 
\[ \textbf{w}((x_1,\dots,x_{n+2})) = \min\{\textbf{w}(x_1), \dots, \textbf{w}(x_{n+2})\}.\] 

Define now the following subsets of $(\Sigma\cap \hat{B}_m)^{n+2}$:
\begin{align*}
\mathscr{F}_{m,0} &:= (\Sigma_{m,0}\cup \Sigma_{m,\varepsilon})^{n+2},\\
\mathscr{F}_{m,0}' &:= \{(x_1,\dots,x_{n+2}) \in (\Sigma\cap B_m')^{n+2} : \text{for some distinct $i,j$,}\\ 
&\qquad\qquad\text{$x_i \in \Sigma_{m,0}$ and $x_j \in \zeta_m\circ G_m(T_1\cup T_2)$}\},
\end{align*}
for $w\in\bigcup_{i=0}^{m-1}\{1,2\}^i$
\begin{align*}
\mathscr{F}_{m,w} &:= (\Sigma_{m,w}\cup \Sigma_{m,w1} \cup \Sigma_{m,w2})^{n+2},
\end{align*}
and for $w\in\bigcup_{i=0}^{m-2}\{1,2\}^i$ (if $m\geq 2$)
\begin{align*}
\mathscr{F}_{m,w}' := \{(x_1,\dots,x_{n+2}) &\in (\Sigma\cap B_m')^{n+2} : \text{$\textbf{w}((x_1,\dots,x_{n+2}))=w$ and}\\
&\text{$x_i \in \zeta_m\circ G_m(T_{w11}\cup T_{w12} \cup T_{w21} \cup T_{w2})$ for some $i$}\}.
\end{align*}

\begin{lemma}
For each $m\in\N$,
\[(\Sigma\cap \hat{B}_m)^{n+2} = \mathscr{F}_{m,0}\cup \mathscr{F}_{m,0}' \cup \bigcup_{i=0}^{m-1}\bigcup_{w\in \{1,2\}^i}\mathscr{F}_{m,w} \cup \bigcup_{i=0}^{m-2}\bigcup_{w\in \{1,2\}^i}\mathscr{F}_{m,w}'.\]
\end{lemma}

\begin{proof}
Let $\textbf{x} = (x_1,\dots,x_{n+2}) \in (\Sigma\cap \hat{B}_m)^{n+2}\setminus \mathscr{F}_{m,0}$. Then there exists $i_1$ such that 
\[ x_{i_1} \in \bigcup_{i=1}^{m-1}\bigcup_{w\in\{1,2\}^i}\Sigma_{w,m} \cup \bigcup_{u\in\{1,2\}^m}\Sigma_{u,m} = \zeta_m\circ G_m (T_1\cup T_2).\]
If there exists $i\neq i_1$ such that $x_i \in \Sigma_{m,0}$, then $\textbf{x}\in \mathscr{F}_{m,0}'$. 

Assume for now on, that no such $i$ exists. That is, for all $i$,
\[ x_i \in \Sigma_{m,\varepsilon} \cup \zeta_m\circ G_m (T_1\cup T_2) = \zeta_m\circ G_m (T_{\varepsilon}).\]
It remains to show that $\textbf{x} \in \bigcup_{i=0}^{m-1}\bigcup_{w\in \{1,2\}^i}\mathscr{F}_{m,w} \cup \bigcup_{i=0}^{m-2}\bigcup_{w\in \{1,2\}^i}\mathscr{F}_{m,w}'$. Let $w = \textbf{w}(\textbf{x})$ and let $i_2$ such that $x_{i_2} \in \Sigma_{m,w}$. In particular, for all $i$,
\[ x_i \in \bigcup_{v\in \{1,2\}^{m-|w|}}\Sigma_{m,wu} = \zeta_m\circ G_m(T_w).\]
We now distinguish two possible cases.

\emph{Case 1.} Suppose that $x_i\in \Sigma_{m,w}\cup \Sigma_{m,w1} \cup \Sigma_{m,w2}$ for all $i$. Then $\textbf{x} \in \mathscr{F}_{m,w}$.

\emph{Case 2.} If the assumption of Case 1 fails, then there exists $j\neq i_1$ such that
\begin{align*} x_j \in \zeta_m\circ G_m(T_w) \setminus (\Sigma_{m,w}&\cup \Sigma_{m,w1} \cup \Sigma_{m,w2})\\ 
&= \zeta_m\circ G_m(T_{w11}\cup T_{w12} \cup T_{w21} \cup T_{w2})
\end{align*}
which implies that $\textbf{x} \in \mathscr{F}_{m,w}'$.
\end{proof}

By Lemma \ref{lem:C2} we have that
\[ E:= \max\{\mathcal{E}_p^{n+2}(\Sigma_{1,0}\cup \Sigma_{1,\varepsilon}),\mathcal{E}_p^{n+2}(\Sigma_{1,\varepsilon} \cup \Sigma_{1,1}\cup\Sigma_{1,2}), \mathcal{E}_p^{n+2}(\Sigma_{2,\varepsilon} \cup \Sigma_{2,1}\cup\Sigma_{2,2})\} < \infty. \]
Inequality \eqref{eq:Cestimate2} follows from four estimates.

First, by Lemma \ref{lem:mengerscale},
\begin{align*}
\sum_{m=1}^{\infty}\int_{\mathscr{F}_{m,0}}\mathcal{K}(\textbf{x})^{p} \, d\mathcal{H}^{n(n+2)}(\textbf{x})&= \sum_{m=1}^{\infty} \mathcal{E}_p^{n+2}(\Sigma_{m,0}\cup \Sigma_{m,\varepsilon})\\
&= \mathcal{E}_p^{n+2}(\Sigma_{1,0}\cup \Sigma_{1,\varepsilon})\sum_{m=1}^{\infty} 2^{-(m^2+5)(p-n(n+2))}\\
&< \infty.
\end{align*}

Second, if $\textbf{x} \in \mathscr{F}_{m,0}'$, then $\diam{\Delta(\textbf{x})} \simeq 2^{-m^2-5}$ and by $n$-regularity,
\begin{align*}
\sum_{m=1}^{\infty}\int_{\mathscr{F}_{m,0}}\mathcal{K}(\textbf{x})^{p} \,d\mathcal{H}^{n(n+2)}(\textbf{x})&\lesssim \sum_{m=1}^{\infty} 2^{p(m^2+5)}\mathcal{H}^{n(n+2)}(\Sigma\cap \hat{B}_m)\\
&\lesssim \sum_{m=1}^{\infty} 2^{-(m^2+5)(n(n+2)-p)}\\
&<\infty.
\end{align*}

Third, by Lemma \ref{lem:mengerscale},
\begin{align*}
\sum_{m=1}^{\infty}\sum_{i=0}^{m-1}\sum_{w\in\{1,2\}^i}\int_{\mathscr{F}_{m,w}}&\mathcal{K}(\textbf{x})^{p} \, d\mathcal{H}^{n(n+2)}(\textbf{x})\\ 
&\lesssim \sum_{m=1}^{\infty}\sum_{i=0}^{m-1}\sum_{w\in\{1,2\}^i} \mathcal{E}_{p}^{n+2}(\Sigma_{m,w}\cup \Sigma_{m,w1} \cup \Sigma_{m,w2})\\
&\lesssim E\sum_{m=1}^{\infty}\sum_{i=0}^{m-1} 2^i (2^{-i}2^{-m^2-5})^{n(n+2)-p}\\
&\lesssim \sum_{m=1}^{\infty} 2^{-(m^2+5)(n(n+2)-p)}\max\{1,2^{m(1-p+n(n+2))}\}\\
&<\infty.
\end{align*}

Fourth, if $\textbf{x} \in \mathscr{F}_{m,w}'$, then $\diam{\Delta(\textbf{x})} \simeq 2^{-|w|}2^{-m^2-5}$ and by $n$-regularity,
\begin{align*}
\sum_{m=2}^{\infty}\sum_{i=0}^{m-2}\sum_{w\in\{1,2\}^i}\int_{\mathscr{F}_{m,w}'}&\mathcal{K}(\textbf{x})^{p} \, d\mathcal{H}^{n(n+2)}(\textbf{x})\\
&\lesssim \sum_{m=2}^{\infty}\sum_{i=0}^{m-2} 2^i (2^{-i}2^{-m^2-5})^{n(n+2)-p}\\
&\lesssim \sum_{m=2}^{\infty} 2^{-(m^2+5)(n(n+2)-p)}\max\{1,2^{m(1-p+n(n+2))}\}\\
&<\infty.
\end{align*}
This completes the proof of \eqref{eq:Cestimate2} and the proof of Theorem \ref{thm:ex}.

\section{Linear approximation of $1$-manifolds}\label{sec:1dim}

The goal of this section is to prove Theorem \ref{prop:LAPfor1mfds2}. Along the way, we also show the following proposition of independent interest, which roughly states that if $\G$ is a topological circle with uniformly small betas, then it is locally a quasicircle in a quantitative fashion

Here and for the rest of this section, given a topological circle $\G\subset \R^n$ and $x,y \in \G$, we denote by $\Gamma(x,y)$ the component of $\Gamma \setminus \{x,y\}$ of smallest diameter. 

\begin{proposition}\label{prop:LAPimpliesBT}
Suppose that $\G \subset \R^n$ is a topological circle with the $(1,\d,R)$-LAP for some $\d\in (0,\frac{1}{8\sqrt{2}})$ and $R>0$.
Then
\begin{equation}\label{eq:1BT}
\diam{\G(x,y)} \leq \frac1{(1-8\sqrt{2}\d)^2}|x-y|, \qquad\text{for all $x,y \in \Gamma$ with $|x-y| < R/4$}.
\end{equation}
\end{proposition}



In particular, \cite{TuVa} implies that a curve $\Gamma$ as in Proposition \ref{prop:LAPimpliesBT} is a quasicircle.

The following remark is elementary.

\begin{remark}\label{rem:geometry}
Let $\ell \subset \R^n$ be a line, let $\delta<1/\sqrt{2}$, let $x\in\ell$, and let $r>0$. Then $\partial B^n(x,r) \cap N_{\d r}(\ell)$ contains exactly two components (which are topological $(n-1)$-balls) such that $\diam{D_i} = 2\delta r$ and $\dist(D_1,D_2) = 2\sqrt{1-\delta^2}r$.
\end{remark}

\begin{lemma}\label{lem:continua}
Let $\delta<1/\sqrt{2}$, let $R>0$, and let $X\subset \R^n$ be a closed set with the $(1,\d,R)$-LAP. Let $\ell \subset \R^n$ be a line containing $x$ such that
 \[ \sup\{\dist(z,\ell) : z\in B^n(x,r)\cap X\} < \delta r.\]
Let $D,D'$ be the two components of $\partial B^n(x,r) \cap N_{\d r}(\ell)$. If $X_1,X_2 \subset X\cap B^n(x,r)$ are two continua that intersect both $D$ and $D'$, then $X_1\cap X_2 \neq \emptyset$.
\end{lemma}

\begin{proof}
Let $X$, $x$, $r$, $\ell$, $X_1$, and $X_2$ be as above. Fix $\lambda \in (\sqrt{2}\delta,1)$. 

Set $\ell_1 := \ell$, $r_1:=r$, $D_1:=D$, and $D_1':=D'$.

We proceed inductively. Assume that for some $k\in \N$ we have defined a line $\ell_k \subset \R^n$ such that
\begin{enumerate}
\item $\sup\{\dist(z,\ell_k) : z\in B^n(x,\lambda^{k-1}r)\cap X\} < \delta \lambda^{k-1}r$,
\item there exist two continua 
\[ X_{1,k}\subset X_1\cap B^n(x,\lambda^{k-1}r) \quad\text{and}\quad X_{2,k}\subset X_2\cap B^n(x,\lambda^{k-1}r)\] 
that both intersect both components $D_{k},D_k'$ of $\partial B^n(x,\lambda^{k-1}r) \cap N_{\d \lambda^{k-1}r}(\ell_k)$. 
\end{enumerate}

By the $(1,\delta,R)$-LAP, there exists a line $\ell_{k+1} \subset \R^n$ containing $x$ such that 
\[ \sup\{\dist(z,\ell_{k+1}) : z\in B^n(x,\lambda^{k}r)\cap X\} < \delta \lambda^{k}r. \]
Let $D_{k+1},D_{k+1}'$ be the two  components of $\partial B^n(x,\lambda^{k}r) \cap N_{\d \lambda^{k}r}(\ell_{k+1})$, and let $D^1,D^2$ be the two components of $\partial B^n(x,\lambda^{k}r) \cap N_{\d \lambda^{k}r}(\ell_{k})$. Then there exist continua $X_{1,k+1}\subset X_{1,k}\cap B^n(x,\lambda^{k}r)$ and $X_{2,k+1}\subset X_2\cap B^n(x,\lambda^{k}r)$ that both intersect $D^1$ and $D^2$.

Since $\delta < 1/\sqrt{2}$ and $\lambda> \sqrt{2}\delta$, we have that $\delta^2\lambda^2 < \lambda^2 - \delta^2$ which, along with Remark \ref{rem:geometry}, gives
\begin{align*}
\dist(D^1,D^2) = 2\sqrt{\lambda^2-\delta^2}\lambda^{k-1}r > 2\delta \lambda^{k}r = \diam{D_{k+1}} = \diam{D_{k+1}'}.
\end{align*}
Therefore, each of $D_{k+1},D_{k+1}'$ intersects exactly one of $D^1,D^2$. Hence, each of $X_{k+1,1},X_{k+1,2}$ intersects both of $D_{k+1},D_{k+1}'$. This proves the inductive step.

Consequently, for any $k\in \N$ we have that $X_1 \cap B^n(x,\lambda^k r) \neq \emptyset$ and $X_2 \cap B^n(x,\lambda^k r) \neq \emptyset$ which gives that $x \in X_1\cap X_2$.
\end{proof}

\begin{lemma}\label{lem:curve-endpoints}
Suppose that $\G \subset \R^n$ is a topological circle with the $(1,\d,R)$-LAP for some $\d\in (0,\frac{1}{8\sqrt{2}})$ and $R>0$. Let $x\in \G$, let $r\in (0,R)$, and let $\ell \subset \R^n$ be a line containing $x$ such that
 \[ \sup\{\dist(z,\ell) : z\in B^n(x,r)\cap \G\} < \delta r.\]
If $\g$ is the component of $\G\cap B^n(x,r)$ that contains $x$, then $\g$ intersects both components $D,D'$ of $\partial B^n(x,r) \cap N_{\d r}(\ell)$.
\end{lemma}

\begin{proof}
Assume, for a contradiction, that $\g$ has endpoints $x_1,x_2$ only on one of $D,D'$; say on $D$. Let $\pi : \R^n \to \ell$ be the orthogonal projection on $\ell$. Without loss of generality, we may assume that $\pi(x_1)$ is between $\pi(x)$ and $\pi(x_2)$. Let $w$ be the midpoint of the line segment $[\pi(x),\pi(x_1)]$ and note that 
\[ \tfrac12r\sqrt{1-\delta^2} \leq |w-x| \leq \tfrac12 r.\] 
Let also $w' \in \g$ such that $\pi(w')=w$. Then
\[ B^n(w',r/3) \subset B^n(w', r(1-\sqrt{\tfrac14 + \d^2})) \subset B^n(x,r)\]
and if $\ell'$ is the line passing from $w$ and parallel to $\ell$, then 
\[ \dist(z,\ell') \leq 2\d r = (6\d)\frac{r}3, \quad\text{for all $z\in \Gamma\cap B^n(w',r/3)$}\]
with $6\d < 1/\sqrt{2}$. 

Now, $\g \cap B^n(w',r/3)$ contains two components, $\g_{1}$ and $\g_{2}$, that each intersect both components of $N_{2\d r}(\ell')\cap \partial B^n(w',r/3)$; one is contained in $\Gamma(x_1,x)$ and the other in $\Gamma(x_2,x)$. These two components cannot intersect, because $x\notin B^n(w',r/3)$. On the other hand, Lemma \ref{lem:continua} says that they must intersect, yielding the contradiction.
\end{proof}

We can now show Theorem \ref{prop:LAPfor1mfds2}.

\begin{proof}[{Proof of Theorem \ref{prop:LAPfor1mfds2}}]
Suppose that $\Gamma \subset\R^n$ is a 1-manifold with the $(1,\d,R)$-LAP for some $\d \in (0,\frac{1}{8\sqrt{2}})$ and $R>0$. Fix $x\in \G$, $r\in (0,R)$, and (by compactness) a line $\ell \subset \R^n$ that contains $x$ such that

\[ \sup\{\dist(z,\ell) : z \in B^n(x,r)\cap \G\} =\beta^1_{\G}(B^n(x,r)) < \d r.  \]
By Lemma \ref{lem:curve-endpoints}, the component $\g$ of $\G\cap B^n(x,r)$ that contains $x$, intersects both components $D,D'$ of $\partial B^n(x,r) \cap N_{\d r}(\ell)$. Setting $\beta = \beta_{\G}^1(x,r)$, elementary geometric calculations give
\begin{align*}
\theta^1_{\G}(B^n(x,r)) &\leq r^{-1}\inf_{z\in \ell\cap B^n(x,r)} \dist(z,\G\cap B^n(x,r))\\ 
&\leq r^{-1}\inf_{z\in \ell\cap B^n(x,r)} \dist(z,\g)\\ 
&\leq \sqrt{2}\sqrt{1- \sqrt{1-\beta^2}}\\
&\frac{\sqrt{2}\beta}{\sqrt{1+ \sqrt{1-\beta^2}}}.
\end{align*}
Using simple algebra,
\begin{align*}
\theta^1_{\G}(B^n(x,r)) \leq \beta\sqrt{\frac{2}{1+ \sqrt{1-\beta^2}}}\leq \beta\sqrt{\frac{2}{1+ \sqrt{1-\d^2}}} &\leq \beta \sqrt{\frac{2}{2- \delta^2}}\\
&\leq \beta \sqrt{1+\delta^2}\\
&\leq \beta(1+\delta^2/2).\qedhere
\end{align*}
\end{proof}

We finish with the proof of Proposition \ref{prop:LAPimpliesBT}.

\begin{proof}[Proof of Proposition \ref{prop:LAPimpliesBT}]
To show \eqref{eq:1BT}, it is enough to show that for all $x,y \in \Gamma$ with $|x-y| < R/2$ and all $z\in\G(x,y)$
\begin{equation}\label{eq:2pt}
\max\{|z-x|,|z-y|\} \leq \frac{1}{1-8\sqrt{2}\d}|x-y|.
\end{equation}
Indeed, let $x,y \in \G$ such that $|x-y|<R/4$ and let $w,z \in \G(x,y)$ such that $|w-z| = \diam{\G(x,y)}$. Without loss of generality, we may assume that $w$ is between $x$ and $z$. By \eqref{eq:2pt} we have that
\[ |x-z| \leq \frac{1}{1-8\sqrt{2}\d}|x-y| < \frac{1}{1-8\sqrt{2}\d}\frac{R}4 < \frac{R}2. \]
Applying \eqref{eq:2pt} twice, 
\[ \diam{\G(x,y)} = |w-z|\leq \frac{1}{1-8\sqrt{2}\d}|x-z| \leq \frac{1}{(1-8\sqrt{2}\d)^2}|x-y|. \]

We now turn to the proof of \eqref{eq:2pt}. Towards a contradiction, assume that there exist $x,y \in \Gamma$ and $z\in \G(x,y)$ such that $|x-z| > (1+\e) |x-y|$ for some $\e> \frac{8\sqrt{2}}{1-8\sqrt{2}\d}$. Let $r=(1+\e)|x-y|$. By the $(1,\d,R)$-LAP, there exists a line $\ell\subset R^n$ containing $x$ such that
\[\sup\{\dist(z,\ell) : z \in B^n(x,r)\cap \Gamma\} \leq \delta r.\]
Since $\d < 1/\sqrt{2}$, the intersection  $\partial B^n(x,r) \cap N_{\d r}(\ell)$ contains exactly two components (two topological $(n-1)$-balls) $D_1$ and $D_2$. By the $(1,\d,R)$-LAP, $\G\cap \partial B^n(x_1,r) \subset D_1\cup D_2$.

Let $\g_1$ and $\g_2$ be the components of $\G\cap B^n(x,r)$ that contain $x$ and $y$ respectively. Since $z\not\in B^n(x,r)$, we have $\g_1 \neq \g_2$. By Lemma \ref{lem:curve-endpoints}, we know that $\g_1$ has endpoints in both $D_1$ and $D_2$. There are two possible cases to consider.

\emph{Case 1.} Suppose that $\g_2$ also has endpoints in both $D_1$ and $D_2$. Then by Lemma \ref{lem:continua}, we get that $\g_1\cap \g_2 \neq \emptyset$ which is a contradiction.

\emph{Case 3.} Suppose that $\g_2$ has endpoints only on one of $D_1,D_2$. We work as in Lemma \ref{lem:curve-endpoints}. Without loss of generality, assume that both endpoints of $\g_2$ are on $D_2$. As with Case 2, we may assume that both endpoints of $\g_2$ (denoted by $y_1,y_2$) are on $D_1$ and we let $\pi : \R^n \to \ell$ be the orthogonal projection on $\ell$. Assume, as we may, that $\pi(y_1)$ is between $\pi(y)$ and $\pi(y_2)$. We have that
\begin{align*}
|\pi(y)-\pi(y_i)| &\geq |x-\pi(y_i)| -|x-\pi(y)|\\ 
&\geq |x-y_i|-|y_i-\pi(y_i)| -|x-y|-|y-\pi(y)|\\ 
&\geq r\left(\frac{\e}{1+\e}-2\d \right).
\end{align*}

Let $w$ be the midpoint of $[\pi(y),\pi(y_i)]$ and let $w'\in \g_2$ such that $\pi(w')=w$. Note that
\begin{align*} 
\dist(w',\partial B^n(x,r)) \geq \dist(w,\partial B^n(x,r)) - |w-w'| &\geq \frac12r\left(\frac{\e}{1+\e}-2\d \right) - \d r\\ 
&= r\left(\frac{\e}{2(1+\e)}-\d \right)\\ 
&> \frac{\e}{4(1+\e)}r
\end{align*}
so
\[ B^n(w',\tfrac{\e}{4(1+\e)}r) \subset B^n(x,r).\]

If $\ell'$ is the line parallel to $x$ passing from $w$, then 
\[ \dist(z,\ell') \leq 2\d r = \left(\frac{8(1+\e)\d}{\e}\right)\frac{\e}{4(1+\e)}r, \quad\text{for all $z\in \Gamma\cap B^n(w',\tfrac{\e}{4(1+\e)}r)$}\]
and by our assumption, $\frac{8(1+\e)\d}{\e} < 1/\sqrt{2}$. Now, $\g_2 \cap B^n(w',\tfrac{\e}{4(1+\e)}r)$ contains two components $\g_{2,1}$ and $\g_{2,2}$ that intersect both components of 
\[ N_{2\d r}(\ell')\cap \partial B^n(w',\tfrac{\e}{4(1+\e)}r).\] By Lemma \ref{lem:continua}, these two components must intersect, which is false.
\end{proof}


\bibliographystyle{alpha}

\bibliography{bibliography}

\end{document}